\documentclass{article}
\usepackage[left=1in,top=1in,right=1in,bottom=1in,head=.1in]{geometry}

\usepackage{natbib}
\usepackage{amsmath}
\usepackage{amssymb}
\usepackage{mathtools}
\usepackage{amsthm}
\usepackage{subcaption}
\usepackage[textsize=tiny]{todonotes}
\usepackage{hyperref}
\usepackage[dvipsnames]{xcolor}
\usepackage{mathtools}

\hypersetup{
colorlinks=true,
citecolor=ForestGreen,
linkcolor=RoyalBlue,   
urlcolor=RubineRed!90!black
}

\usepackage{cleveref}
\usepackage{aliascnt}
\usepackage{graphicx}
\usepackage{constants}
\usepackage{mathtools}
\usepackage{xcolor}
\usepackage{bm}
\usepackage{bbm}

\newcommand\eps\varepsilon
\newcommand\hbbeta{\bm{\hat{\beta}}}
\newcommand\tbbeta{\bm{\tilde{\beta}}}

\DeclareMathOperator{\tr}{tr}
\DeclareMathOperator{\prox}{prox}

\DeclareMathOperator{\E}{\mathbb{E}}
\DeclareMathOperator{\PP}{\mathbb{P}}

\DeclareMathOperator*{\argmin}{arg\,min}

\DeclareMathOperator{\oper}{op}
\DeclareMathOperator{\F}{F}

\newcommand{\Op}{\mathcal{O}_{\mathrm{p}}}
\newcommand{\op}{o_{\mathrm{p}}}

\newcommand{\by}{\bm y}
\newcommand{\bX}{\bm X}
\newcommand{\bN}{\bm N}
\newcommand{\bM}{\bm M}
\newcommand{\bV}{\bm V}

\newcommand{\bD}{\bm D}
\newcommand{\bW}{\bm W}

\newcommand{\bP}{\bm P}
\newcommand{\bQ}{\bm Q}
\newcommand{\bI}{\bm I}
\newcommand{\bA}{\bm A}
\newcommand{\mL}{\mathcal{L}}
\newcommand{\bbeta}{\bm \beta}
\newcommand{\R}{\mathbb{R}}
\newcommand{\bx}{\bm x}
\newcommand{\bpsi}{\bm \psi}
\newcommand{\tbpsi}{\bm{\tilde{\psi}}}
\newcommand{\bh}{\bm h}
\newcommand{\bz}{\bm z}

\newcommand{\be}{\bm e}
\newcommand{\bb}{\bm b}
\newcommand{\bv}{\bm{v}}
\newcommand{\bu}{\bm{u}}
\newcommand{\bw}{\bm{w}}
\newcommand{\bg}{\bm{g}}
\newcommand{\q}{q}
\newcommand{\hbb}{\hat{\bb}}

\numberwithin{equation}{section}
\newtheorem{theorem}{Theorem}[section]
\newaliascnt{proposition}{theorem}

\newtheorem{proposition}[proposition]{Proposition}
\aliascntresetthe{proposition}

\newaliascnt{lemma}{theorem}
\newtheorem{lemma}[lemma]{Lemma}
\aliascntresetthe{lemma}

\newaliascnt{remark}{theorem}
\newtheorem{remark}[remark]{Remark}
\aliascntresetthe{remark}

\newaliascnt{assumption}{theorem}
\newtheorem{assumption}[assumption]{Assumption}
\aliascntresetthe{assumption}

\begin{document}

\title{Asymptotics of resampling without replacement in robust and
logistic regression}



\author{
Pierre C. Bellec\thanks{Department of Statistics, Rutgers University. 
Email: \href{mailto:pierre.bellec@rutgers.edu}{pierre.bellec@rutgers.edu}}
\and
Takuya Koriyama\thanks{Booth School of Business, University of Chicago. 
Email: \href{mailto:tkoriyam@uchicago.edu}{tkoriyam@uchicago.edu}}
}

\date{\today}
\maketitle

\begin{abstract}
  This paper studies the asymptotics of resampling without
  replacement in the proportional regime where dimension $p$
  and sample size $n$ are of the same order.
  For a given dataset $(\bX,\by)\in\R^{n\times p}\times \R^n$
  and fixed subsample ratio $\q\in(0,1)$,
  the practitioner samples independently of $(\bX,\by)$
  iid subsets $I_1,...,I_M$ of $\{1,...,n\}$
  of size $\q n$
  and trains estimators $\hbb(I_1),...,\hbb(I_M)$
  on the corresponding subsets of rows of $(\bX,\by)$.
  Understanding the performance of the bagged estimate
  $\bar\bb = M^{-1}\sum_{m=1}^M \hbb(I_m)$, for instance
  its squared error, requires
  us to understand correlations between two distinct
  $\hbb(I_m)$ and $\hbb(I_{m'})$ trained on different
  subsets $I_m$ and $I_{m'}$.
  
  In robust linear regression and logistic regression, we characterize
  the limit in probability of the correlation between two
  estimates trained on different subsets of the data.
  The limit is characterized as the unique solution of a
  simple nonlinear equation. We further provide data-driven estimators
  that are consistent for estimating this limit. These estimators
  of the limiting correlation allow us to estimate the squared error
  of the bagged estimate $\bar\bb$, and for instance perform
  parameter tuning to choose the optimal subsample ratio $\q$.
  As a by-product of the proof argument, we obtain the limiting
  distribution of the bivariate pair $(\bx_i^T \hbb(I_m), \bx_i^T \hbb(I_{m'}))$ for observations $i\in I_m\cap I_{m'}$, i.e., for observations
  used to train both estimates.
\end{abstract}

\maketitle

\tableofcontents

\section{Introduction}
\label{s:intro}

This paper studies the performance of bagging estimators
trained on subsampled, overlapping datasets in the context
robust linear regression and logistic regression. 

\subsection{M-estimation in the proportional regime}
We consider an M-estimation problem in the proportional
regime where sample size \( n \) and dimension \( p \) are of the same order:
Throughout the paper, \( \delta > 1 \) is a fixed constant and
the ratio
\begin{equation}
    n/p = \delta
    \label{regime}
\end{equation}
is held fixed as \( n,p\to+\infty \) simultaneously.
The practitioner collects data \( (y_i,\bx_i)_{i\in[n]} \)
with scalar-valued responses \( y_i \) and feature vectors \( \bm x_i \in \R^p \). For a given subset of observations \( I\subset [n] \),
an estimator \( \hat{\bb}(I) \) is trained on the subset of
observations \( (y_i,\bx_i)_{i\in I} \) using an optimization problem
of the form
\begin{equation}
    \hat{\bb}(I) = \argmin_{\bb \in \R^p} \sum_{i\in I} {\ell}_{y_i}(\bx_i^\top \bb) 
    \label{hbbI_intro}
\end{equation}
where for each \( i\in[n] \), the loss
\( {\ell}_{y_i}(\cdot)  \) is convex and depends implicitly on the response \( y_i \).
We will focus on two regression settings:
robust linear regression and Generalized Linear Models (GLM),
including logistic regression.
In robust regression, the response is of the form 
\begin{equation}
    \label{y_i_robust}
    y_i = \bm x_i^T \bbeta^* + \eps_i
\end{equation}
for some possibly heavy-tailed noise $\eps_i$ independent of $\bx_i$.
In this case the loss ${\ell}_{y_i}$ in \eqref{hbbI_intro}
is given by
\begin{equation}
    {\ell}_{y_i}(u) = \rho(y_i-u)
    \label{ell_i_robust}
\end{equation}
where $\rho$ is a deterministic function, for instance
the Huber loss $\rho (u) = \int_0^{|u|}\min(1,t)dt$
or its smooth variants, e.g., $\rho(u)=\sqrt{1+u^2}$.
The asymptotics of the performance of \eqref{hbbI_intro}
with $I=\{1,...,n\}$ and the loss \eqref{ell_i_robust} in robust regression
in the proportional regime \eqref{regime} are now well understood
\cite{el2013robust,donoho2016high,el2018impact,thrampoulidis2018precise}
as we will review in \Cref{sec:robust}.
A typical example of GLM to which our results apply is the case of 
binary logistic regression,
where \( {\ell}_{y_i} \) in \eqref{hbbI_intro} is the negative log-likelihood  
\begin{equation}
    {\ell}_{y_i}(u) = \log(1+e^u) - u y_i,
    \qquad y_i\in \{0,1\}
    \label{logi_loss}
\end{equation}
which is now also well understood for $I=[n]$ in \eqref{hbbI_intro} \cite{sur2018modern,candes2020phase}. Related results will be reviewed
in \Cref{sec:logistic}.
The goal of the present paper is to study the performance
of bagging several estimators of the form \eqref{hbbI_intro}
obtained from several subsampled datasets $I_1,...,I_M$. \\

\subsection{Bagging estimators trained on subsampled datasets without replacement}

Let \( M>0 \) be a fixed integer, held fixed as \( n,p\to+\infty \).
The practitioner then samples \( M \) subsets of \( [n] \)
according to the uniform distribution on all subsets of \( [n] \)
of size \( \q n \) for some $q\in (0,1]$, that is, 
\begin{equation}
    I_1,...,I_M
    \sim^{\text{iid}} \text{Unif}\{I\subset [n] : |I|=\q n \}.
    \label{I_1_I_M}
\end{equation}
Each \( I_m \) thus samples a subset of \( [n] \) of size \( \q n \)
without replacement 
 and the set of indices \( I_1,...,I_M \)
are all independent. Throughout this paper, we will refer this procedure as \textit{sampling without replacement}. 
While the set of indices are independent, the corresponding 
subsampled datasets $(\bx_i, y_i)_{i \in I_m}$ and $(\bx_i, y_i)_{i \in I_{m'}}$ 
are not independent as soon as there is some overlap in the sense
\( I_m\cap I_{m'} \ne \emptyset \).

\begin{remark}\label{rm:gepmetric_dist}
If \( I_m \) and \( I_{m'} \) are independent according
to \eqref{I_1_I_M} then \( |I_m \cap I_{m'}| \) follows a hyper-geometric
distribution with mean \( q^2 n \), and by Chebychev's inequality
using the explicit formula for the variance of hyper-geometric distributions,
\(|I_m \cap I_{m'}|/n \to^P q^2  \) as \( n\to+\infty \) while \( \q \)
is held fixed. Thus, not only is the intersection non-empty 
with high-probability, but it is of order \( n \).
\end{remark}
The goal of the paper is to understand the performance
of bagging the corresponding subsampled estimates:
with the notation \( \hbb(I) \) in \eqref{hbbI_intro}
and \( I_1,...,I_M \) in \eqref{I_1_I_M}, the practitioner
constructs the bagged estimate
\begin{equation}
    \label{bagging_intro}
    \bar\bb=
    \frac1 M \sum_{m=1}^M \hbb(I_m).
\end{equation}

\subsection{Related work}
In the proportional regime \eqref{regime}, 
\cite{peter1995learning, anders1997statistical} derived the limiting generalization error for ensembles of estimators $\hat{\bb}(I_m)$ whose distribution follows a Gibbs measure proportional to $\exp(-\mathcal{L}_{I_m}(\bb;\lambda)/T)$, where $T>0$ is the temperature parameter and $\mathcal{L}_{I_m}(\bb;\lambda)$ denotes the $\ell_2$-regularized empirical risk: $\mathcal{L}_{I_m}(\bb;\lambda) = \sum_{i\in I_m} (y_i - \bx_i^\top \bb)^2 + \lambda \|\bb\|_2^2$. Based on this result, they showed via numerical simulations that for a fixed temperature $T>0$, the ensemble estimator with a fixed regularization level $\lambda>0$ and optimally tuned subsample size $|I|$ can achieve strictly lower generalization error than a single estimator $\hat{\bb}([n])$ trained on the full dataset with an optimally tuned regularization parameter. 
Bagging as a generally applicable principle was introduced
in \cite{breiman1996bagging,breiman2001using} and early analysis in low-dimensional regimes 
were performed in \cite{buhlmann2002analyzing} among others.
In the proportional regime \eqref{regime}, 
\cite{lejeune2020implicit} demonstrated the role of bagging
as an implicit regularization technique when the base learners
$\hbb(I_m)$ are least-squares estimates.
Bagging Ridge estimators was studied in \cite{du2023subsample,patil2023bagging}
who characterized the limit of the squared error of \eqref{bagging_intro}
using random matrix theory.
The implicit regularization power of bagging
in the proportional regime is again seen in 
\cite{patil2023bagging,du2023subsample}, where it is shown that the optimal risk among Ridge estimates
can also be achieved by bagging Ridgeless estimates and optimally choosing the subsample size.
Estimating the risk of a bagged estimate such as \eqref{bagging_intro}
for regularized least-squares estimates is done in
\cite{patil2023bagging,du2023subsample,bellec2023corrected}.
The risk of bagging random-features estimators, trained on the full dataset
but with each base learner having independent weights within the random
feature activations, is characterized in \cite{loureiro2022fluctuations}.
Most recently, \cite{clarte2024analysis} studied the limiting equations of
several resampling schemes including bootstrap and resampling without
replacement, and characterized self-consistent equations for the limiting risk
of estimators obtained by minimization of the negative log-likelihood and an
additive Ridge penalty. However, the specific nonlinear systems we study (\eqref{eta_equation_robust} and \eqref{eta_equation_logi}) do not explicitly appear in their work, which instead focuses on bias and variance functionals associated with particular resampling strategies. The results in \cite{clarte2024analysis} build on the general AMP framework and the state evolution analysis developed in \cite[Lemmas B.3 and B.5]{loureiro2022fluctuations}, extending the foundational work of \cite{bayati2011lasso}. Their approach relies on the existence and uniqueness of solutions to the limiting system of equations, which is guaranteed under strong convexity assumptions (e.g., with a Ridge penalty) but was not established in the case without such an assumption until the present paper appeared.

\subsection*{Organization}
We will first study and state our main results for robust regression
in \Cref{sec:robust}.
\Cref{sec:logistic} extends the results to logistic regression.
Numerical simulations are provided in \Cref{sec:simu_robust} in robust
regression and in \Cref{subsec:simu_logistic} in logistic regression.
The main results are proved in
\Cref{sec:proof} simultaneously
for robust linear regression 
and logistic regression.
\Cref{sec:auxilliary_lemma} contains several auxiliary
lemmas used in the proof in \Cref{sec:proof}.

\subsection*{Notation}
For vectors $\|\cdot\|$ or $\|\cdot\|_2$ is the Euclidean norm, while $\|\cdot\|_{\oper}$ and $\|\cdot\|_{\F}$ denote the operator norm
and Frobenius norm of matrices.
The arrow $\to^P$ denotes convergence in probability
and $\op(1)$ denotes any sequence of random variables
converging to 0 in probability.
The stochastically bounded notation $\Op(r_n)$ for $r_n>0$ denotes
a sequence of random variables such that for any $\eta>0$, there exists $K>0$
with $\PP(\Op(r_n)>K r_n)\le \eta$.

\section{Robust regression}
\label{sec:robust}

This section focuses on robust regression in the linear model
\eqref{y_i_robust}, where the noise variables $\eps_i$
are possibly heavy-tailed. Throughout the paper, our working assumption
for the robust linear regression setting is the following.
\begin{assumption}
    \label{assum:robust}
    Let $\q\in(0,1),\delta>0$ be constants such that $\q\delta > 1$
    and $n/p = \delta$ as $n,p\to+\infty$.
    Let $\bbeta^*\in\R^p$.
    Assume that $(\bx_i,y_i)_{i\in[n]}$ are iid with
    $y_i = \bx_i^T\bbeta^* + \eps_i$ and $\eps_i$ independent of $\bx_i\sim N(\bm{0}_p,\bI_p)$ {satisfying $\PP(\eps_i\ne 0)> 0$.}
    Assume that the loss is
    ${\ell}_{y_i}(u)= \rho(y_i - u)$ for a twice-continuously differentiable
    function $\rho$ {with $\argmin_{x\in\R}\rho(x)=\{0\}$}
    as well as $|\rho'(t)|\le 1$ and $0<\rho''(t)\le 1$
    for all $t\in\R$.
\end{assumption}

Robust loss functions that meets \Cref{assum:robust} include the pseudo-Huber loss $\rho(t) = \sqrt{1 + t^2}$ and its scaled variant $\rho_\lambda(t) = \{\lambda^2/(1+\lambda)\} \cdot \rho(t / \lambda)$ for any $\lambda>0$. In contrast, the standard Huber loss $\rho(t) = \int_0^{|t|} \min(1, x) dx$ does not meet the requirement $\inf_{t\in\mathbb{R}}\rho''(t) > 0$ imposed in \Cref{assum:robust}. 

Nevertheless, we emphasize that the most essential and fundamental condition on the robust loss function $\rho$ is the Lipschitz continuity, namely, $\sup_{t\in \mathbb{R}}|\rho'(t)|\le 1$. Indeed, an unregularized M-estimator fitted by a Lipschitz convex loss has a finite risk limit for any noise distribution, while for any non-Lipschitz convex loss function, there exists a heavy-tailed noise under which the risk diverges (see Section 2 and Proposition E.2 in \cite{koriyama2023analysis}). 
On the other hand, the condition $\inf_{t\in\mathbb{R}}\rho''(t)>0$ is primarily an artifact of our proof technique, and we verify by numerical simulation that our main theorem holds for the Huber loss (see \Cref{sec:simu_robust}). We expect that the condition $\inf_{t\in\mathbb{R}} \rho''(t)>0$ can be relaxed, by a smoothing argument that adds a vanishing Ridge penalty term to the
optimization problem \eqref{hbbI_intro}, 
as explained in \cite[Section 1.3]{bellec2023error} and \cite[Section B.2.1]{koriyama2024precise}.

With $|I| = qn$ and $\delta = n/p$, the assumption $q\delta (=|I|/p) >1$ is necessary for the unregularized M-estimator $\hat\bb (I) \in \argmin_{\bb\in \R^p} \sum_{i\in I} \rho(y_i-\bx_i^\top\bb)$ to be well-defined. 
The condition $\PP(\eps_i \ne 0) > 0$ is assumed to avoid the trivial case where the perfect recovery $\hat{\bb}(I)=\bbeta_*$ holds with probability $1$. 
Indeed, if $\PP(\eps_i \ne 0)=0$, then combined with $\{0\}=\argmin_x\rho(x)$ for the convex loss $\rho$, this gives 
$\rho'(\eps_i)=0$ for all $i\in I$ with probability $1$, so that  $\sum_{i\in I} \bx_i \rho'(\eps_i)=\bm{0}_p$ with probability $1$. 
By the KKT condition for the unregularized M-estimator, this means $\hat{\bb}(I) = \bbeta_*$ with probability $1$.

\subsection{A review of existing results in robust linear regression}
\label{s:intro_robust}
The seminal works \cite{donoho2016high,el2013robust,karoui2013asymptotic,el2018impact} characterized the performance of robust M-estimation
in the proportional regime \eqref{regime}.
For a convex loss $\rho:\R\to\R$ and ${\ell}_{y_i}$ as in \Cref{assum:robust},
these works characterized the limiting squared risk
$\|\hbb(\{1,...,n\})-\bbeta^*\|^2$
of an estimator $\hbb(\{1,...,n\})$,
trained on the full dataset, i.e.,
taking $I=\{1,...,n\}$ in \eqref{hbbI_intro}.
In particular, 
\cite{donoho2016high,el2013robust,karoui2013asymptotic,el2018impact,thrampoulidis2018precise} show that under the design of $(\bx_i,y_i)$ given in
\Cref{assum:robust}, the squared risk of $\hbb(\{1,...,n\})$
converges in probability to a constant, and this constant is found
by solving a system of two nonlinear equations with two unknowns.
If a subset $I\subset [n]$ of size $|I|=\q n$ is used to train
\eqref{hbbI_intro}, simply changing $\delta=n/p$ to $\delta\q=|I|/p$,
these results imply the convergence in probability
$\|\hbb(I)-\bbeta^*\|^2\to^P \sigma^2$ where $(\sigma,\gamma)$
is the solution to the system
\begin{align}
    \label{Robust1}
    \tfrac{\sigma^2}{\delta\q}&= \E[
    (\sigma G - \prox[\gamma{\ell}_{y}](\sigma G))^2
    ]
    \\
    1
    -
    \tfrac{1}{\delta\q} 
    &= 
    \sigma^{-1}
        \E[G
    \prox[\gamma{\ell}_y](\sigma G)
    ]
    \label{Robust2}
\end{align}
where \( G \sim N(0,1) \) is independent of ${y}$ and $y=^d y_i$, i.e., $y$ follows the same distribution as any marginal of the response vector $\by = (y_i)_{i \in [n]}$. 
Above, \( \prox[f](x_0)=\argmin_{x\in\R}(x_0-x)^2/2 + f(x) \) denotes the proximal operator of a convex function
\( f \) for any \( x_0\in \R \).
The system \eqref{Robust1}-\eqref{Robust2} was predicted
in \cite{el2013robust} using a heuristic leave-one-out argument.
Early rigorous results \cite{donoho2016high,karoui2013asymptotic,el2018impact} assumed either $\rho$ is strongly convex (\cite{donoho2016high})
or added an additive strongly convex Ridge penalty to the M-estimation problem
(\cite{karoui2013asymptotic,el2018impact});
\cite{thrampoulidis2018precise} generalized such results
without strong convexity.

We now subsample without replacement, obtaining iid subsets
$I_1,...,I_M$ as in \eqref{I_1_I_M}.
For each $m=1,...,M$ the theory above applies individually
to $\hbb(I_m)$. In particular $\|\hbb(I_m)-\bbeta^*\|^2\to^P \sigma^2$.
By expanding the square, the squared L2 error of the average
$\bar\bb$ in \eqref{bagging_intro}  is given by
\begin{equation}
    \|\bar\bb-\bbeta^*\|^2
    =
    \frac{1}{M^2}
    \sum_{m=1}^M
    \|\hbb(I_m)-\bbeta^*\|^2
    +
    \frac{1}{M^2}
    \sum_{m=1}^M
    \sum_{m'=1: m'\ne m}^M
    (\hbb(I_m)-\bbeta^*)^T
    (\hbb(I_{m'})-\bbeta^*).
    \label{risk_bagged}
\end{equation}
Since previous works established that 
$\|\hbb(I_m)-\bbeta^*\|^2\to^P \sigma^2$, the first term above
is clearly $\sigma^2/M$. The crux of the problem is thus
to characterize
the limit in probability, if any,
of each term 
$(\hbb(I_m)-\bbeta^*)^T
(\hbb(I_{m'})-\bbeta^*)$
in the second term inside the double sum.

\subsection{A glance at our results}
Since $\rho$ in \eqref{ell_i_robust} is Lipschitz and differentiable,
the system \eqref{Robust1}-\eqref{Robust2}
admits a unique solution (\cite{koriyama2023analysis}). 
Let \( (\sigma,\gamma ) \) be the solution
to this system (since only the solution to \eqref{Robust1}-\eqref{Robust2} is of interest,
we denote its solution by \( (\sigma,\gamma) \) without extra subscripts
for brevity). 

The key to understanding the performance of the aforementioned
bagging procedure \eqref{bagging_intro} and, for instance,
characterizing the limits of $\| \bar\bb - \bbeta^*\|^2$,
is the following equation
with unknown \( \eta\in[-1,1] \):
\begin{align}
\eta&=\frac{{\q^2}\delta}{\sigma^2}
        \E\Bigl[
        \Bigl(\sigma G - \prox[\gamma {\ell}_{y}](\sigma G)\Bigr)
        \Bigl(\sigma G - \prox[\gamma {\ell}_y](\sigma \tilde G)\Bigr)
        \Bigr]
\label{eta_equation_robust}, \quad 
\begin{pmatrix}
    G \\
    \tilde G
\end{pmatrix}
\sim N\Bigl(0_2,
\begin{pmatrix}
    1 & \eta \\
    \eta & 1
\end{pmatrix}
\Bigr)
\end{align}
with $y=^d y_i$ as in \eqref{Robust1}-\eqref{Robust2} and $(G, \tilde G)$ being independent of $y$.
Using \eqref{Robust1}, the above equation can be equivalently rewritten as
\begin{equation}
    \eta = F(\eta)
    \quad \text{where}\quad  F(\eta) \equiv \q \frac{
        \E\bigl[
        \bigl(\sigma G - \prox[\gamma {\ell}_{y}](\sigma G)\bigr)
        \bigl(\sigma G - \prox[\gamma {\ell}_{y}](\sigma \tilde G)\bigr)
        \bigr]
    }{
        \E[(\sigma G- \prox[\gamma{\ell}_y](\sigma G))^2
        ]
    } 
    \label{def-F}
\end{equation}
since $\E[(\sigma G- \prox[\gamma{\ell}_y](\sigma G))^2
        ] = \sigma^2/(\delta \q)$ in the denominator by \eqref{Robust1}.
This shows that any solution \( \eta \) must satisfy
\( |\eta|\le \q \) by the Cauchy-Schwarz inequality.

We will show in the next section that this equation in \( \eta \)
has a unique solution. Our main results imply a close relationship
between the solution \( \eta \) of \eqref{eta_equation_robust}
and the bagged estimates, in particular \eqref{intro_inner_product_P}
satisfies
\begin{equation}
    \bigl(\hbb(I_m)-\bbeta^*\bigr)^T 
    \bigl(\hbb(I_{m'}) - \bbeta^*\bigr) \to^P \eta \sigma^2.
\end{equation}
For two distinct and fixed \( m \ne m' \), the solution \( \eta \) further characterizes the joint distribution
of two predicted values \( \bx_i^T\hbb(I_m) \) and
\( \bx_i^T\hbb(I_m) \) with \( i\in I_m\cap I_{m'} \), by
showing the existence of \( (G_i,\tilde G_i) \) as in \eqref{eta_equation_robust},
independent of \( ({\ell}_i,U_i) \) and such that
\begin{align*}
        \bx_i^T\hbb(I_m)
        &=\prox[\gamma{\ell}_{y_i}](\sigma G_i)
    + \op(1), \quad 
        \bx_i^T\hbb(I_{m'}) =
        \prox[\gamma{\ell}_{y_i}](\sigma\tilde G_i)
    + \op(1)
\end{align*}

\subsection{Existence and uniqueness of solutions to the fixed-point
equation}

\begin{proposition}\label{lm:eta_exist}
    The function \( F \) in \eqref{def-F}
    is non-decreasing and \( \q \)-Lipschitz with \( 0\le F(0)\le q\le 1 \). 
    The equation
    \( \eta = F(\eta) \) has a unique solution \( \eta\in[0,\q] \).
\end{proposition}
\begin{proof}
We may realize $\tilde G$ as
$\tilde G = \eta G + \sqrt{1-\eta^2}Z$
where $Z,G$ are iid $N(0,1)$ independent of ${\ell}_i$.
For any Lipschitz continuous
function $f$ with $\E[f(G)^2]<+\infty$, the map 
$\varphi: \eta \in [-1,1] \mapsto \E[f(G)f(\tilde G)]=\E[f(G)f(\eta G +\sqrt{1-\eta^2}Z)]\in\R$ has derivative
\begin{equation}
\varphi'(\eta) = \E[f'(G)f'(\tilde G)].
\label{varphi-d}
\end{equation}
See \Cref{lm:varphi_derivative} for the proof. 
In our case, this implies that the function
\eqref{def-F}
has derivative
\begin{equation}
F'(\eta) = 
\q^2\delta
\E\Bigl[
    \Bigl(1-\prox[\gamma{\ell}_y]'(\sigma G)\Bigr)
    \Bigl(1-\prox[\gamma{\ell}_y]'(\sigma \tilde G)\Bigr)
\Bigr]
    .
\label{F'}
\end{equation}
Since $\prox[\gamma{\ell}_y]$ is nondecreasing and 1-Lipschitz
for any convex function ${\ell}_y:\R\to\R$,
each factor inside the expectation belongs to $[0,1]$
and $0\le F'(\eta)$ holds.
By bounding from above the second factor,
\begin{align*}
F'(\eta) 
\le
\q^2 \delta
\E\bigl[
    1-\prox[\gamma{\ell}_y]'(\sigma G)
\bigr]
= \q^2 \delta (\q\delta)^{-1} = \q
\end{align*}
thanks to \eqref{Robust2} and Stein's {formula (or integration by parts)}
for the equality.
This shows $0\le F'(\eta) \le \q < 1$
so that $F$ is a contraction and admits a unique solution
in $[-1,1]$. 

We now show that the solution must be in $[0,\q]$.
The definition \eqref{def-F} gives $F(1)=\q$ as $\PP(G=\tilde G)=1$
when $\eta=1$.
Now we verify $F(0)\ge 0$. 
If $\eta=0$ then $(G, \tilde{G}, y)$ are independent and $G=^d\tilde{G}$ so by the tower property of conditional expectations,
\begin{align*}
    F(0) 
    &= \frac{\q^2\gamma^2\delta}{\sigma^2} \E\Bigl[\E\bigl[\bigl(\sigma G-\prox[\gamma {\ell}_y](\sigma G)\bigr) \mid y\bigr]^2\Bigr]
    \ge 0.
\end{align*}
{Since $0\le F(0)\le F(1)\le \q < 1$, the unique fixed-point must belong to $[0,\q]$.}
\end{proof}
\subsection{Main results in robust regression}

For any $I\subset[n]$ with $|I|=\q n = \q\delta p$,
the M-estimator $\hbb(I)=\argmin_{\bb\in\R^p}\sum_{i\in I} {\ell}_{y_i}(\bx_i^T\bb)$
satisfies the convergence in probability
\begin{equation}
    \|\hbb(I)-\bbeta^*\|^2 \to^P \sigma^2,
    \qquad
    \frac{1}{|I|}\sum_{i\in I} \Bigl({\ell}_{y_i}'(\bx_i^T\hbb(I))\Bigr)^2
    \to^P \frac{\sigma^2}{\gamma^2 \q \delta}.
    \label{eq_convergence}
\end{equation}
The first convergence in probability was proved by many authors,
e.g., \cite{el2013robust,donoho2016high,el2018impact,thrampoulidis2018precise}.
The second can be obtained using the CGMT of \cite{thrampoulidis2018precise},
see for instance \cite[Theorem 2]{loureiro2021learning}.
We will take the convergence in probability
\eqref{eq_convergence} for granted in our proof.

\begin{theorem}
    \label{thm:robust}
    Let \Cref{assum:robust} be fulfilled.
    Let $I,\tilde I$ be independent and uniformly distributed over all 
    subsets of $[n]$ of size $\q n$.
    Then
    \begin{equation}
        \label{eq:thm_eta_robust}
        (\hbb(I)-\bbeta^*)^T(\hbb(\tilde I)-\bbeta^*)
         \to^P \sigma^2 \eta,
         \qquad
        \frac{(\hbb(I)-\bbeta^*)^T(\hbb(\tilde I)-\bbeta^*)}
             {\|(\hbb(I)-\bbeta^*)\|_2 \|(\hbb(\tilde I)-\bbeta^*)\|_2}
             \to^P \eta
    \end{equation}
    where $\eta\in[0,\q]$ is the unique solution to \eqref{eta_equation_robust}.
    Furthermore, $\eta$ and $\eta\sigma^2$ can be consistently estimated
    in the sense
    \begin{equation}
        \label{eq:thm_estimation}
    \frac{\hat\gamma(I) \hat\gamma(\tilde I)}{p} \sum_{i\in I \cap \tilde I}{\ell}_{y_i}'\Bigl(\bx_i^T\hbb(I)\Bigr)
    {\ell}_{y_i}'\Bigl(\bx_i^T\hbb(\tilde I)\Bigr)
    \to^P \eta\sigma^2,
    \quad
    \frac{\hat\gamma(I)^2}{p} \sum_{i\in I}
    {\ell}_{y_i}'\Bigl(\bx_i^T\hbb(I)\Bigr)^2
    \to^P \sigma^2
    \end{equation}
    where 
    $$
    \hat \gamma(I)
    =p/\Bigl[
        \sum_{i\in I} {\ell}_{y_i}''(\bx_i^T\hbb(I))
    - {\ell}_{y_i}''(\bx_i^T\hbb(I))^2 \bx_i^T\bigl(\sum_{l\in I}\bx_l{\ell}_{y_l}''(\bx_l^T\hbb(I))\bx_l^T\bigr)^{-1} \bx_i\Bigr]
    $$
    Finally, for any $i\in I\cap \tilde I$,
    there exists $(G_i,\tilde G_i)$ jointly normal as in
    \eqref{eta_equation_robust} with $\E[G_i\tilde G_i]=\eta$ such that
    \begin{equation}
        \label{eq:thm_distribution_robust}
        \max_{i\in I\cap\tilde I}
        \E\Bigl[ 1 \wedge
        \Big\|
        \begin{pmatrix}
            \bx_i^T\hbb(I)\\
            \bx_i^T\hbb(\tilde I)
        \end{pmatrix}
        -
        \begin{pmatrix}
            \prox[\gamma{\ell}_{y_i}](\sigma G_i)\\
            \prox[\gamma{\ell}_{y_i}](\sigma \tilde G_i)\\
        \end{pmatrix}
        \Big\|_2
        \mid (I,\tilde I)
        \Bigr]
        \to^P 0.
    \end{equation}
\end{theorem}

\Cref{thm:robust} is proved in \Cref{sec:proof}.
It provides three messages. First,
\eqref{eq:thm_eta_robust} states that
the correlation 
$(\hbb(I)-\bbeta^*)^T(\hbb(\tilde I)-\bbeta^*)$
between two estimators trained in independent subsets $I,\tilde I$
both of cardinally $\q n$ converges to the unique solution $\eta$
of \eqref{eta_equation_robust}.
A direct consequence is that the squared risk of the bagged
estimate \eqref{risk_bagged} satisfies
\begin{equation}
\|\bar\bb-\bbeta^*\|^2 \to^P \sigma^2/M + (1-1/M)\sigma^2\eta.
\label{risk_bagged_limit}
\end{equation}
Second,
both terms in this risk decomposition of the bagged estimate $\bar\bb$ can be 
estimated using \eqref{eq:thm_estimation} averaged over all pairs
$(I_m,I_{m'})_{m\ne m'}$, that is,
\begin{equation*}
    \frac{1}{M^2}
    \sum_{m\ne m'}
    \frac{\hat\gamma(I_m) \hat\gamma(\tilde I_{m'})}{p} \sum_{i\in I_{m'} \cap \tilde I_m}{\ell}_{y_i}'\Bigl(\bx_i^T\hbb(I_{m'})\Bigr)
    {\ell}_{y_i}'\Bigl(\bx_i^T\hbb(\tilde I)\Bigr)
    \to^P \Bigl(1-\frac1M\Bigr)\eta\sigma^2,
\end{equation*}
and 
$\frac{1}{M^2}\sum_{m=1}^M\frac{\hat\gamma(I_m)^2}{p} \sum_{i\in I_m}
    {\ell}_{y_i}'(\bx_i^T\hbb(I_m))^2
    \to^P \sigma^2/M$.
These estimators let us estimate the risk of the bagged estimate
\eqref{risk_bagged_limit}, for instance to choose an optimal subsample size
$\q\in(0,1)$, or to choose a large enough constant $M>0$ 
so that \eqref{risk_bagged_limit} is close to the large-$M$ limit
given by $\sigma^2\eta$.
At a high level, these estimators take the form of an inner product of “residuals,”-specifically $\sum_{i\in I\cap \tilde I} \ell_{y_i}' (\bx_i^\top \hat\bb(I)) \ell_{y_i}'(\bx_i^\top \hat\bb(\tilde I))$-
followed by observable adjustments through the factors $\hat\gamma(I)$ and $\hat\gamma(\tilde{I})$. This result is complement to the Corrected Generalized Cross-Validation (CGCV) developed in \cite[equation (13)]{bellec2023corrected}, which similarly constructs a risk estimator as an adjusted inner product of residuals, in the context of regularized least-squares estimators.

As shown in \Cref{fig:huber}, resampling and bagging is sometimes
beneficial but not always. Whether the curve $\q \mapsto \sigma^2\eta$
is U-shaped and minimized at some $\q^*<1$ (i.e., bagging is beneficial)
depends on the interplay between 
the oversampling ratio $\delta=n/p$, 
the distribution of the noise
$\eps_i$ and 
the robust loss function $\rho$ used in \eqref{hbbI_intro}.
In \Cref{fig:huber}, we observe that if $\eps_i/\tau$ has
t-distribution with 2 degrees of freedom and $\delta=5$,
subsampling is not beneficial for $\tau=1$ but
becomes beneficial for $\tau\ge 1.5$.
The generality of this phenomenon is unclear at this point.

The third message of \Cref{thm:robust} is the characterization
of the limiting bivariate distribution of $(\bx_i^T\hbb(I),\bx_i^T\hbb(\tilde I))$ for an observation $i\in I \cap \tilde I$ used to train both
$\hbb(I)$ and $\hbb(\tilde I)$.
The convergence \eqref{eq:thm_distribution_robust}
implies that 
$(\bx_i^T\hbb(I),\bx_i^T\hbb(\tilde I))$
converges to the distribution of
$(\prox[\gamma{\ell}_{y_i}](\sigma G_i),\prox[\gamma{\ell}_{y_i}](\sigma \tilde G_i))$ weakly. Here $(G_i,\tilde G_i)$ has the multivariate normal distribution as in \eqref{eta_equation_robust}. 

The setting of resampling without replacement in the proportional regime
of the present paper is also studied in the recent paper \cite{clarte2024analysis}.
There are some significant differences between our contributions and 
\cite{clarte2024analysis}.
First, an additive Ridge penalty is imposed in \cite{clarte2024analysis}
and multiple resampling schemes are studied,
while our object of interest is the unregularized M-estimator
\eqref{hbbI_intro} with a focus on resampling without replacement.
The simple
fixed-point equation \eqref{eq:thm_eta_robust} does not appear explicitly
in \cite{clarte2024analysis}, which instead focuses on self-consistent
equations satisfied by bias and variance functionals \cite[(16)]{clarte2024analysis}
of the specific resampling scheme under study.
Another distinctive contribution of the present paper is the proposed estimator
\eqref{eq:thm_estimation} which can be used to optimally tune the subsample size,
and the proof that the equation \eqref{eta_equation_robust} admits a unique solution.
The use of an additive Ridge penalty brings strong
convexity to the optimization problem and simplifies the
analysis, as observed in \cite{el2013robust}; in this case this makes the
analysis \cite[(212)-(218)]{loureiro2022fluctuations} based on
\cite{bayati2011lasso} readily applicable.

\subsection{Numerical simulations in robust regression}
\label{sec:simu_robust}
Let us verify \Cref{thm:robust} with numerical simulations.
Throughout this section, we focus on the Huber loss 
\begin{equation*}
    \rho(t)=
\begin{cases}
t^2/2 &\text{ if }|t|<1,\\
|t|-1/2& \text{ if }|t|\ge 1.
\end{cases}
\end{equation*}
The oversampling ratio $\delta=n/p$ is fixed to $5$.
First, we plot $\eta$ and $\sigma^2\eta$ as functions of ${q}\in [1/\delta, 1]$ for different noise scales: we change the noise distribution as $\{\text{scale}\}\times \text{t-dist (df=2)}$, $\text{scale}\in \{1, 1.5, 2, 5, 10\}$. The left figures in \Cref{fig:huber} imply that the curve $q\mapsto \eta$ is nonlinear. Note that the dashed line is the affine line $\q\mapsto (q-\delta^{-1})/(1-\delta^{-1})$. 
More interestingly, the larger the noise scale is, the larger the nonlinearity is. In the right figures in \Cref{fig:huber}, we observe that the plot ${q}\mapsto\eta\alpha^2$ takes a U-shape curve when the noise scale is sufficiently large. Note that similar results are obtained for ensembles of Ridge estimators in \cite{anders1997statistical}. Interestingly, \Cref{fig:huber} suggests that as the scale of noise distribution increases, sub-sampling is eventually beneficial in the sense that the limit of \eqref{risk_bagged_limit} as $M\to+\infty$ is smaller than the squared error of a single estimate trained on the full dataset.
This phenomenon also occurs when the noise distribution has a finite variance (see \Cref{subsec:change_noise_dist}).

Next, we compare in simulations the correlation and the inner product
     with their theoretical limits $(\eta, \eta\sigma^2)$ as in \eqref{eq:thm_eta_robust}, as well as the estimator in \eqref{eq:thm_estimation}. 
Here, the noise distribution is fixed to $3\cdot \text{t-dist(df=2)}$ with $(n, p)=(5000, 1000)$ and ${100}$ repetitions. \Cref{fig:huber_compare} implies that the correlation and product are approximated well by the corresponding theoretical values and estimates.

We have also conducted the same experiment for the pseudo-Huber loss $\rho(x)=\sqrt{1+x^2}$ in \Cref{subsec:pseudo_huber} and verified the validity of \Cref{thm:robust}.

\begin{figure}
    \centering
    \begin{subfigure}[b]{0.49\textwidth}
        \centering
        \includegraphics[width=\textwidth]{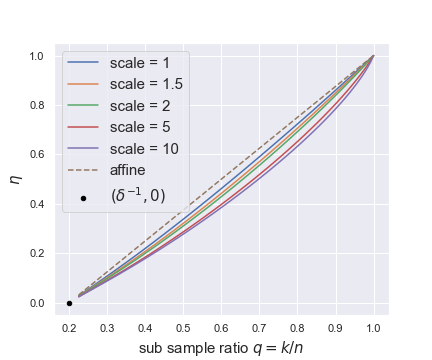}
    \end{subfigure}
    \begin{subfigure}[b]{0.49\textwidth}
        \centering
        \includegraphics[width=\textwidth]{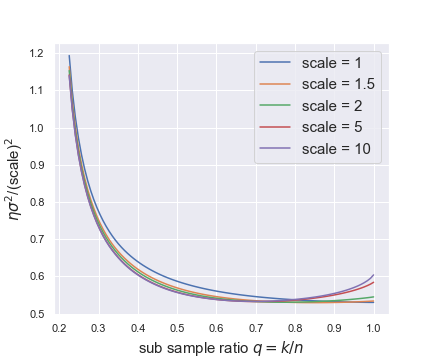}
    \end{subfigure}
    \begin{subfigure}[b]{0.49\textwidth}
        \centering
        \includegraphics[width=\textwidth]{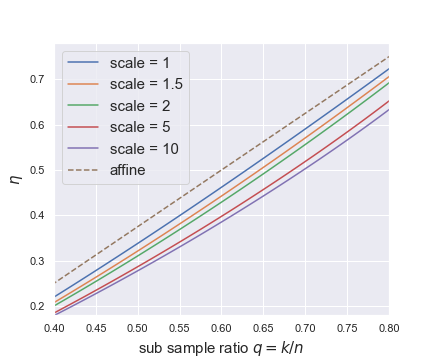}
    \end{subfigure}
    \begin{subfigure}[b]{0.49\textwidth}
        \centering
        \includegraphics[width=\textwidth]{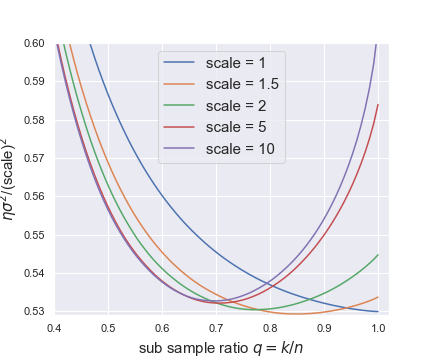}
    \end{subfigure}
    \caption{Plot of ${\q} \mapsto \eta$ and ${\q}\mapsto \sigma^2\eta$
        obtained by solving \eqref{eta_equation_robust} numerically.
        {Different noise distributions are given by} $(\text{scale})\times \text{t-dist (df=2)}$, for scale$\in\{1,{ 1.5, 2, 5,} 10\}$. The dashed line is the affine line $\q\mapsto (q-\delta^{-1})/(1-\delta^{-1})$. 
    {The bottom plots zoom in on a specific region of the top plots.}
    }
    \label{fig:huber}
\end{figure}

\begin{figure}
    \centering
    \begin{subfigure}[b]{0.49\textwidth}
        \centering
        \includegraphics[width=\textwidth]{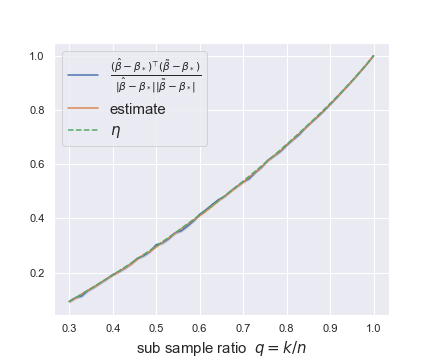}
    \end{subfigure}
    \begin{subfigure}[b]{0.49\textwidth}
        \centering
        \includegraphics[width=\textwidth]{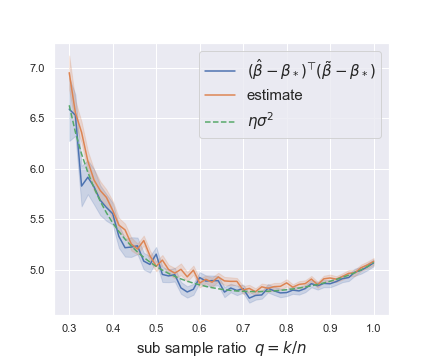}
    \end{subfigure}
    \caption{Comparison of simulation results, theoretical curves
            obtained by solving \eqref{eta_equation_robust} numerically,
    and estimate constructed by \eqref{eq:thm_estimation}. Here, the noise distribution is fixed to $3\times \text{t-dist(df=2)}$ and $(n, p)=(5000, 1000)$.}
    \label{fig:huber_compare}
\end{figure}

\section{Resampling without replacement in logistic regression}
\label{sec:logistic}

\subsection{A review of existing results in logistic regression}
\label{s:intro_logistic}

Let \( \nu >0, \q\in(0,1], \delta>1 \) be fixed constants.
If a single estimator \( \hbb(I) \) is trained with \eqref{hbbI_intro}
on a subset of observations \( I\subset [n] \) with \( |I|/n=\q \)
for some constant \( \q\in(0,1] \) held fixed as \( n,p\to+\infty \),
the behavior of \( \hbb(I) \) is now well-understood
when \( (y_i,\bx_i)_{i\in [n]} \) are iid with \( \bx_i \sim N(\bm 0_p,\bI_p)\) normally
distributed and the conditional distribution \( y_i \mid \bm x_i \)
following a logistic model of the form
\begin{equation}
    \PP\Bigl(y_i= 1 \mid \bm x_i\Bigr)  
    = \frac{1}{1+\exp(-\bm x_i^T\bbeta^*)}
    = \frac{1}{1+\exp(- \nu \bm x_i^T\bw)}
\end{equation}
where \( \bbeta^* \) is a ground truth with \( \|\bbeta^*\|=\nu\), and 
\( \bm w= \bbeta^*/\nu \) is the projection of \( \bbeta^* \) on the unit
sphere.
In this logistic regression model,
the limiting behavior of \( \hbb(I) \) with the logistic loss \eqref{logi_loss}
{trained using $|I|=(\delta\q) p$ samples}
is characterized as follows:
there exists a monotone continuous function \( h(\cdot) \)
(with explicit expression given in \cite{candes2020phase})
such that:

\( \bullet  \)
If \( \delta\q < h(\nu) \) then the logistic MLE 
        \eqref{hbbI_intro} does not exist with high-probability.

\( \bullet  \)
If \( \delta\q > h(\nu) \) then there exists a unique
\cite{sur2018modern}
solution \( (\sigma_*,a_*,\gamma_*) \) to the following
the low-dimensional system of equations
\begin{align}
    \label{1}
    \tfrac{\sigma^2}{\delta\q}&= \E[
    (a U + \sigma G - \prox[\gamma{\ell}_y](a U + \sigma G))^2
    ], \\
    0 &= \E[
    (a U + \sigma G - \prox[\gamma{\ell}_y](a U + \sigma G))
    ],
    \label{2}
    \\
    1
    -
    \tfrac{1}{\delta\q} 
      &= \sigma^{-1}
    \E[G
        \prox[\gamma{\ell}_y]'(a U + \sigma G)
    ]
    \label{3}
\end{align}
where $G\sim N(0,1)$ is independent of $(y, U)$ and $(y, U)=^d (y_i, \bx_i^\top\bw)$ for any $i$. 
Above, \( \prox[f](x_0)=\argmin_{x\in\R}(x_0-x)^2/2 + f(x) \) denotes the proximal operator of any convex function
\( f \) for any \( x_0\in \R \).
In this region \( \{\delta\q > h(\nu)\} \) where the above system
admits a unique solution
$(a,\sigma,\gamma)$,
the logistic MLE \eqref{hbbI_intro}
exists with high-probability and
the following convergence in probability holds,
\begin{align}
    \bm w^T \hbb(I) &\to^P a  ,
    \label{limit-a_*}
    \\
    \qquad
    \|(\bm I_p-\bw\bw^T)\hbb(I)\|^2
    &\to^P
    \sigma^2,
    \label{limit-sigma_2}
    \\
    \frac{1}{|I|}\sum_{i\in I}{\ell}_{y_i}'\Bigl(\bx_i^T\hbb(I)\Bigr)^2
    &\to^P \frac{\sigma^2}{ \gamma^2\q\delta}.
    \label{limit-sigma2_by_q_gamma_2_delta}
\end{align}
by \cite{sur2018modern,salehi2019impact} for the first two lines and 
\cite[Theorem 2]{loureiro2021learning} for the third.
Further results are obtained in \cite{candes2020phase,sur2018modern,zhao2020asymptotic}, including asymptotic normality results for individual
components \( \hat b_j \) of \eqref{hbbI_intro}. 
Note that the 3-unknowns system \eqref{1}-\eqref{3} is stated
in these existing works
after integration of the distribution of \( y \). We choose the equivalent
formulation \eqref{1}-\eqref{3} without integrating the conditional
distribution of \( y \) as the form \eqref{1}-\eqref{3} is closer
to \eqref{Robust1}-\eqref{Robust2} from robust regression,
and closer to the quantities naturally appearing in our proofs.
In \Cref{sec:proof}, this common notation is useful to prove the
main results simultaneously for robust linear regression and logistic
regression.

While the limit in probability of the correlation
\( \bar\bb^T\bbeta^* \) can be deduced directly from \eqref{limit-a_*},
the case of Mean Squared Error (MSE) \( \|\bar\bb - \bbeta^*\|^2 \) 
or the correlation $\bar\bb^T\bbeta^*$
is more subtle.
To see the crux of the problem,
recall \( \bw = \bbeta^*/\|\bbeta^*\| \), define
\( \bP = (\bm I_p-\bw\bw^T) \) for brevity,
and 
consider the decomposition:
\begin{align}
    \|\bar\bb - \bbeta^*\|^2
    &=
    \bigl(\bm w^T(\bar\bb - \bbeta^*)\bigr)^2
    + \|\bP\bar \bb\|^2=      \bigl(\bm w^T(\bar\bb - \bbeta^*)\bigr)^2
    + \frac{1}{M^2}\sum_{m,m'=1}^M
    \hbb(I_m)^T\bP \hbb(I_{m'}).   \label{intro_inner_product_P}
\end{align}
In order to characterize the limit of the MSE of \( \bar\bb \),
or to characterize the limit of the normalized correlation
\( \|\bar\bb\|^{-1} \bar\bb^T\bbeta^* \), we need to first understand
the limit of the inner product $\hbb(I_m)^T\bP \hbb(I_{m'})$, 
where \( \hbb(I_m) \) and \( \hbb(I_{m'}) \) are trained on two
subsamples \( I_m \) and \( I_{m'} \) with non-empty intersection.
This problem happens to be almost equivalent to the corresponding
one in robust regression, and we will prove the following
result and \Cref{thm:robust} simultaneously.

\subsection{Main results for logistic regression}

\begin{assumption}
    \label{assum:logi}
    Let $\q\in(0,1),\nu>0,\delta>0$ be constants such that $\q\delta>h(\nu)$
    as $n/p = \delta$ as $n,p\to+\infty$
    with $\bbeta^*\in\R^p$ satisfying $\|\bbeta^*\|=\nu$.
    Assume that $(\bx_i,y_i)_{i\in[n]}$ are iid with
    $y_i\in \{0,1\}$ following the logistic model 
    $\PP(y_i=1\mid \bx_i)=1/(1+\exp(-\bx_i^T\bbeta^*))$.
    Assume that the loss ${\ell}_{y_i}$ is the usual binary logistic loss
    given by \eqref{logi_loss}.
\end{assumption}
In other words, we assume a logistic model with parameters  on the side of the
phase transition where the MLE exists with high-probability.
In this regime, the system \eqref{1}-\eqref{3} admits a unique
solution $(a,\sigma,\gamma)$ and the convergence in probability
\eqref{limit-a_*}-\eqref{limit-sigma2_by_q_gamma_2_delta} holds.

\begin{proposition}\label{prop:eta_exist_logit}
    Under \Cref{assum:logi}, the equation
    \begin{equation}
        \begin{aligned}
        \eta=\frac{{\q^2}\delta\gamma^2}{\sigma^2}
        \E\Bigl[
            {\ell}_y'\Bigl(\prox[\gamma {\ell}_y](a U + \sigma G)\Bigr)
            {\ell}_y'\Bigl(\prox[\gamma {\ell}_y](a U + \sigma \tilde G)\Bigr)
        \Bigr], \quad 
        \begin{pmatrix}
            G \\
            \tilde G
        \end{pmatrix}
        \sim N\Bigl(0_2,
            \begin{pmatrix}
                1 & \eta \\
                \eta & 1
            \end{pmatrix}
        \Bigr)
        \end{aligned}
        \label{eta_equation_logi}
    \end{equation}
    with unknown $\eta$ admits a unique solution $\eta\in[0,\q]$.
Above, $(G,\tilde G)$
are independent of $(U, y)$ and $(U, y)=^d (\bx_i^\top \bw, y_i)$. 

\end{proposition}
We omit the proof since it is exactly same as the proof of \Cref{lm:eta_exist}. 
Similarly to robust regression in \Cref{thm:robust}, the solution
$\eta$ to \eqref{eta_equation_logi} characterizes
the limit in probability of the correlation $\hbb(I_m)^T\bP \hbb(I_{m'})$, 
the estimator \eqref{eq:thm_estimation} is still valid for estimating
$\eta\sigma^2$, and finally we can characterize the joint distribution
of  two predicted values \( \bx_i^T\hbb(I_m) \) and
\( \bx_i^T\hbb(I_m) \) for an observation \( i\in I_m\cap I_{m'} \)
appearing in both datasets.

\begin{theorem}
    \label{thm:logi}
    Let \Cref{assum:logi} be fulfilled 
    and let 
    $\bP=\bI_p - \bbeta^* \frac{1}{\|\bbeta^*\|^2} \bbeta^*{}^T$.
    Let $I,\tilde I$ be independent and uniformly distributed over all 
    subsets of $[n]$ of size $\q n$.
    Then
    \begin{equation}
        \hbb(I)\bP \hbb(\tilde I)
         \to^P \sigma^2 \eta,
         \qquad
        \frac{\hbb(I)^T \bP \hbb(\tilde I)}
             {\|\bP \hbb(I)\|_2 \|\bP \hbb(\tilde I)\|_2}
             \to^P \eta
             \label{eq:thm_eta_logi}
    \end{equation}
    where $\eta\in[0,\q]$ is the unique solution to \eqref{eta_equation_logi}.
    Furthermore, $\eta$ and $\eta\sigma^2$ can be consistently estimated
    in the sense that \eqref{eq:thm_estimation} holds.
    Finally, for any $i\in I\cap \tilde I$,
    there exists $(G_i,\tilde G_i)$ as in
    \eqref{eta_equation_robust}, independent of $(y_i,U_i) = (y_i, \bx_i^\top \bbeta_*/\|\bbeta_*\|)$ such that
    \begin{equation}
             \label{eq:thm_eta_distribution_logi}
        \max_{i\in I\cap\tilde I}
        \E\Bigl[ 1 \wedge
        \Big\|
        \begin{pmatrix}
            \bx_i^T\hbb(I)\\
            \bx_i^T\hbb(\tilde I)
        \end{pmatrix}
        -
        \begin{pmatrix}
            \prox[\gamma{\ell}_{y_i}](aU_i + \sigma G_i)\\
            \prox[\gamma{\ell}_{y_i}](aU_i + \sigma \tilde G_i)\\
        \end{pmatrix}
        \Big\|_2
        \mid (I,\tilde I)
        \Bigr]
        \to^P 0.
    \end{equation}
    
\end{theorem}

\subsection{Numerical simulations in logistic regression}
\label{subsec:simu_logistic}
Similarly to \Cref{sec:simu_robust}, we check the accuracy of
\Cref{thm:logi} with numerical simulations. 
 Here, $(n, p)$ is fixed to $(5000, 500)$ so that $\delta=n/p=10$. For each signal strength $\|\bbeta_*\|\in\{1,2\}$, 
we compute the correlation 
and the inner product (see \eqref{eq:thm_eta_distribution_logi}) as we change the sub-sampling ratio $q=k/n\in[0.4, 1]$ and the estimate constructed by \eqref{eq:thm_estimation}. We perform $100$ repetitions. The theoretical limits $(\eta, \sigma^2\eta)$ 
are obtained by solving \eqref{eta_equation_logi} numerically.
\Cref{fig:logistic} shows that the theoretical curves ($q\mapsto \eta$ and $q\mapsto \sigma^2\eta$) match with the correlation and the inner product.
The estimator \eqref{eq:thm_estimation} is accurate for medium to large
subsample ratio $\q$, but appears slightly biased upwards for small values
of $\q$. 
The source of this slight upward bias is unclear,
although possibly due to the finite-sample nature of the simulations $(p=500)$.

In all simulations for logistic regression that we have performed,
the curve $\q\mapsto \eta$ is affine, as in the left plot 
in \Cref{fig:logistic}. The reason for this is unclear to us at this point
and this appears to be specific logistic regression; for instance the curve
$\q\mapsto\eta$ in \Cref{fig:huber}
for robust regression are clearly non-affine. Furthermore, the curve $\q \mapsto \sigma^2 \eta$ is monotonic, in contrast to the robust regression case, where it exhibits a U-shaped behavior under high noise levels. 
To further investigate the effect of the subsample ratio $\q$ on the risk $\sigma^2 \eta$, we present additional numerical simulations in \Cref{sec:additional_simulation_logit}, which reveals that the risk curve $\q \mapsto \sigma^2 \eta$ becomes U-shaped when the aspect ratio is much larger and the signal strength is small.

\begin{figure}
    \centering
    \begin{subfigure}[b]{0.49\textwidth}
        \centering
        \includegraphics[width=\textwidth]{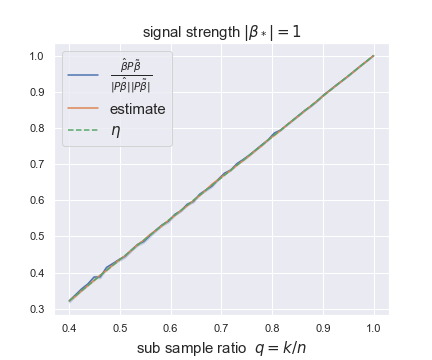}
    \end{subfigure}
    \begin{subfigure}[b]{0.49\textwidth}
        \centering
        \includegraphics[width=\textwidth]{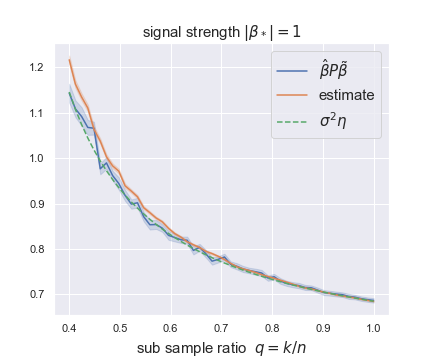}
    \end{subfigure}
    \begin{subfigure}[b]{0.49\textwidth}
        \centering
        \includegraphics[width=\textwidth]{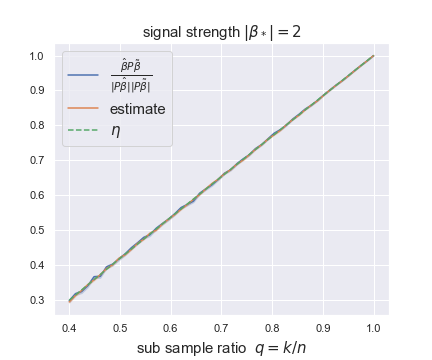}
    \end{subfigure}
    \begin{subfigure}[b]{0.49\textwidth}
        \centering
        \includegraphics[width=\textwidth]{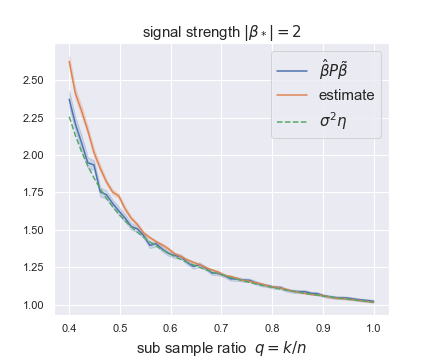}
    \end{subfigure}
    \caption{Comparison of simulation results, theoretical curves
        obtained by solving \eqref{eta_equation_logi} numerically,
    and estimate constructed by \eqref{eq:thm_estimation}, with $(n,p)$ fixed to $(5000, 500)$.} 
    \label{fig:logistic}
\end{figure}

\section{Proof of the main results}
\label{sec:proof}

We prove here \Cref{thm:logi,thm:robust} simultaneously
using the following notation:
\begin{itemize}
    \item In Robust regression (\Cref{thm:robust}), set $a=0$, let $(\sigma,\gamma)$
        be the unique solution to \eqref{Robust1}-\eqref{Robust2},
        let $\bbeta^*=0$ without loss of generality thanks to translation
        invariance; by the linear response $y_i = \bx_i^\top \bbeta_* + \eps_i$ from \Cref{assum:robust} and the change of variable $\bb \mapsto \bh = \bb-\bbeta_*$, we have $
            \hat{\bb}(I)-\bbeta_* = \hat\bh(I)$ with $\hat\bh(I) \in \argmin_{\bh} \sum_{i\in I} \rho(\bx_i^\top \bh + \eps_i)$, which does not depend on the signal $\bbeta_*$. Furthermore, let  $\bP = \bI_p$ and $U_i=0$.
    \item
        In logistic regression (\Cref{thm:logi}),
        let $(a,\sigma,\gamma)$ be the unique solution to \eqref{1}-\eqref{3},
        let $\bP=\bI_p - \bw\bw^\top$ for $\bw=\beta^*/\|\bbeta^*\|$,
        and let $U_i=\bx_i^T\bw$.
        Here, $\bX\bP$ is independent of $(y_i, U_i)_{i\in[n]}$.
\end{itemize}
Thanks to $\|\bX/\sqrt n\|_{\oper}\to^P 1+\delta^{-1/2}$ and
\eqref{eq_convergence} or \eqref{limit-sigma_2}-\eqref{limit-a_*},
we have $\|\bX\hbb(I)\|/\sqrt{I} \le K$ for $K=2 \q^{-1/2}(1+\delta^{-1/2})(a^2+\sigma^2)^{1/2}$ with probability approaching one.
Thus $\PP(\hbb(I)=\hbbeta(I))\to 1$ for $\hbbeta(I)$ in \eqref{reg},
so we may argue with $\hbbeta = \hbbeta(I)$.
Similarly for $\tilde I$ we have $\PP(\hbb(I)=\hbbeta(I))\to 1$
for $\hbbeta(\tilde I)$ in \eqref{reg},
and we may argue with $\tbbeta = \hbbeta(\tilde I)$.
Let also $\bpsi,\tbpsi$ be defined in \Cref{lm:derivative_formula}
(in particular, we have $\psi_i=0$ of $i\notin I$
and $\psi_i = -{\ell}_{y_i}(\bx_i^T\hbb(I))$ in the high-probability
event $\hbb(I)=\hbbeta$,
and similarly for $\tbpsi,\hbb(\tilde I),\tbbeta$.)

By \Cref{lem:SOS} and \Cref{lm:relation_trV_trA_gamma}
from the auxiliary lemmas, we have 
$$ 
p \hat{\bbeta}^\top\bP\tilde{\bbeta} 
= \gamma^2 \bpsi^\top  \tilde \bpsi
+ \Op(\sqrt n)
$$
where $\bpsi^\top  \tilde \bpsi = \sum_{i\in I \cap \tilde I}\psi_i\tilde \psi_i$.
With $n/p=\delta$ and
$|I\cap \tilde I| = n\q^2 + \Op(n^{1/2})$ thanks to
the explicit formulae for the
expectation and variance of the hyper-geometric distribution,
we have
\begin{equation}
\tilde{\bbeta}^\top \bP \hat{\bbeta} = \delta \q^2\gamma^2 \bpsi^T\tbpsi /|I\cap \tilde I|
+ \Op(n^{-1/2}).
\label{eq:key}
\end{equation}
By the Cauchy--Schwarz inequality and the concentration 
of sampling without replacement (see  \Cref{lm:sample_without_rep_psi} for details), 
the absolute value of $\bpsi^T\tbpsi /|I\cap \tilde I| = \sum_{i\in I\cap\tilde I}\psi_i\tilde \psi_i/|I\cap \tilde I|$ is smaller than
\begin{align*}
  \Bigl(\frac{1}{|I\cap \tilde I|}\sum_{i\in I\cap\tilde I}\tilde \psi_i^2
  \Bigr)^{1/2}
  \Bigl(\frac{1}{|I\cap \tilde I|}\sum_{i\in I\cap\tilde I}\psi_i^2
  \Bigr)^{1/2}
  \le
  \Bigl(\frac{1}{|\tilde I|}\sum_{i\in \tilde I}\tilde \psi_i^2
  \Bigr)^{1/2}
  \Bigl(\frac{1}{|\tilde I|}\sum_{i\in I}\psi_i^2
  \Bigr)^{1/2}
  +\op(1)
  = \frac{\sigma^2}{\q \delta\gamma^2}  + \op(1)  
\end{align*}
thanks to \eqref{eq_convergence} (in robust regression)
or \eqref{limit-sigma2_by_q_gamma_2_delta} (in logistic regression)
for the last equality.
Combined with \eqref{eq:key}, 
we have proved
$$
|\tilde{\bbeta}^\top\bP\hat{\bbeta}| \le \delta \q^2\gamma^2 \frac{\sigma^2}{\q\delta\gamma^2} + \op(1)=  \q \sigma^2 + \op(1).
$$
Let $\bar \E$ be the conditional
expectation given $({I,\tilde I},\bX\bbeta^*,\by)$
(In robust regression, $\bbeta^*=0$ so $\bar\E$ is the conditional expectation
given $\{{I,\tilde I}, (\eps_i)_{i\in[n]}\}$). By the Gaussian Poincar\'e inequality, one can show the following concentration (see \Cref{eq:application_Gaussian_Poincare}) $\tilde{\bbeta}^\top\bP\hat{\bbeta}  = \bar{\E}[\tilde{\bbeta}^\top\bP\hat{\bbeta}] + \Op(n^{-1/2})$. 
 Combined with the previous result $|\tilde{\bbeta}^\top\bP\hat{\bbeta}| \le \q \sigma^2 + \op(1)$, we obtain the following:
\begin{equation}
\bar{\eta}\equiv \sigma^{-2} \bar{\E}[\tilde{\bbeta}^\top \bP\hat{\bbeta}]
\quad\text{ satisfies }\quad
\begin{cases}
\bar \eta =
\tilde{\bbeta}^\top \bP \hat{\bbeta}/\sigma^2
+ \Op(n^{-1/2}),
\\
|\bar \eta |\le \q + \op(1).
\end{cases}
\label{concentration-eta-bar}
\end{equation} 
Similarly, by \Cref{eq:application_Gaussian_Poincare} we have the concentration $\bar \E[ \bpsi^T \tbpsi ] /{|I\cap \tilde I|}
=
\bpsi^T \tbpsi/{|I\cap \tilde I|} 
+ \Op(n^{-1/2})$.
Combined with $\tilde{\bbeta}^\top \bP \hat{\bbeta} = \delta \q^2\gamma^2 \bpsi^T\tbpsi /|I\cap \tilde I|
+ \Op(n^{-1/2})$ from \eqref{eq:key}  and $\bar \eta =
\tilde{\bbeta}^\top \bP \hat{\bbeta}/\sigma^2
+ \Op(n^{-1/2})$ from \eqref{concentration-eta-bar}, we get 
\begin{equation}
\bar \eta = 
 \frac{\delta \q^2 \gamma^2}{\sigma^2}
 \frac{1}{|I\cap \tilde I|}\bar{\E}\Bigl[
\sum_{i\in I\cap\tilde I}\psi_i\tilde \psi_i
\Bigr]
+ {\Op(n^{-1/2})}
\label{ineq_eta_bar}
\end{equation}
For an overlapping observation $i\in I\cap \tilde I$,
using \Cref{lm:derivative_formula} and the moment inequality in 
\Cref{lem:W}
conditionally on $(I,\tilde I, \bX\bbeta^*,\by)$ and $(\bx_l)_{l\ne i}$, 
applied to the standard normal $\bP\bx_i + \bw Z$ (for $Z\sim N(0,1)$ independent of everything else) and
 $\bW=[\bP\hat{\bbeta}|\bP\tilde \bbeta]\in\R^{p\times 2}$, 
 we find
 for the indicator function $\mathbb I\{i\in I\cap \tilde I\}$ that
 $$\mathbb I\{i\in I\cap \tilde I\}
 \E[\text{LHS}_i \mid I,\tilde I]
\le  C \E[\sum_{j=1}^p \|\frac{\partial \hat{\bbeta}}{\partial x_{ij}}\|^2 + \|\frac{\partial \tilde \bbeta}{\partial x_{ij}}\|^2]$$
where
$$
\text{LHS}_i
\eqqcolon
\Bigl\|
\begin{pmatrix}
    \bx_i^\top\bP\hat\bbeta - \tr[\bP \bA]\psi_i - (\hat{\bbeta}^\top \bP\bA\bX^\top \bD)\be_i \\
    \bx_i^\top\bP\tilde\bbeta - \tr[\bP \tilde \bA]\tilde \psi_i - (\tilde{\bbeta}^\top\bP\tilde\bA\bX^\top \tilde\bD)\be_i
\end{pmatrix}
-
(\bW^\top\bW)^{1/2} \bg_i
\Bigr\|^2 
$$
for all $i\in I\cap \tilde{I}$ with $\bg_i\sim N(\bm{0}_2, \bI_2)$. 
After summing over $i\in I\cap\tilde I$ and using
\eqref{bound_frobenius_norm_derivatives},
we get $\sum_{i\in I\cap\tilde I}\E[\text{LHS}_i \mid I,\tilde I] \le C$ some
constants $C$ independent of $n,p$, 
and hence 
$\sum_{i\in I\cap\tilde I}\text{LHS}_i = \Op(n^{-1})$.

Using \eqref{limit-sigma_2} in logistic regression
or \eqref{eq_convergence} in robust regression,
we know 
$\|\bP\hat\bbeta\|^2\to^P \sigma^2$ and similarly $\|\bP\tilde\bbeta\|^2\to^P\sigma^2$,
as well as $\tr[\bP\bA]\to^P\gamma$,
and $\tr[\bP\tilde\bA]\to^P\gamma$
by \Cref{lm:relation_trV_trA_gamma}.
Using the Lipschitz
inequality for the matrix square root
$\|\sqrt{\bM}-\sqrt{\bN}\|_{\oper}
\le \|(\sqrt{\bM}+\sqrt{\bN})^{-1}\|_{\oper}
\|\bM - \bN\|_{\oper}$
for positive definite matrices
$\bN,\bM$ (see \cite{van1980inequality} or 
\cite[Problem X.5.5]{bhatia2013matrix})
which follows from $\bx^T(\sqrt{\bM}+\sqrt{\bN})\bx\lambda
= \bx^T(\bM-\bN)\bx$ for any unit eigenvector $\bx$ of $\sqrt{\bM}-\sqrt{\bN}$
with eigenvalue $\lambda$,
here we get
\begin{equation}
    \label{39}
\Big\|
\begin{pmatrix}
    1 & \bar\eta \\
    \bar\eta & 1 \\
\end{pmatrix}^{-1}
\Big\|_{\oper}
= \frac{1}{1-\bar\eta}
\le \frac{2}{1-q}
\text{ and }
    \Big\|
    \sigma
\begin{pmatrix}
    1 & \bar\eta \\
    \bar\eta & 1 \\
\end{pmatrix}^{1/2}
- (\bW^\top\bW)^{1/2}
\Big\|_{\oper}
= \op(1)
\end{equation}
on the event $|\bar\eta|\le (1+q)/2<1$ 
which has probability approaching one thanks to \eqref{concentration-eta-bar}.
Using the moment bounds 
\eqref{eq:moment_bound} to bound from above 
{
$\sum_{i=1}^n((\hat{\bbeta}^\top \bP\bA\bX^\top \bD)\be_i)^2
= \|\bD\bX\bA\bP\hbbeta\|^2$,
}
we find
$$
\frac{1}{n}
\sum_{i\in I\cap \tilde I}
\|
\begin{pmatrix}
    \bx_i^\top\bP\hat\bbeta   - \gamma\psi_i \\
    \bx_i^\top\bP\tilde\bbeta - \gamma\tilde \psi_i
\end{pmatrix}
-
\sigma
\begin{pmatrix}
    1 & \bar\eta \\
    \bar\eta & 1 \\
\end{pmatrix}^{1/2}
\bg_i
\|^2 = \op(1) + \op(1)\frac 1 n \sum_{i=1}^n\bigl(\|\bg_i\|^2+\psi_i^2+\tilde \psi_i^2\bigr),
$$
and thanks to
$n^{-1}\sum_{i=1}^n(\|\bg_i\|^2+\psi_i^2+\tilde \psi_i^2)= \Op(1)$,
the previous display converges to 0 in probability.
Since
$\bx_i^T\bP\hbbeta = \bx_i^T\hbbeta - U_i \bw^T\hbbeta$ for $U_i = \bx_i^T \bm{w} =^d N(0,1)$
and given $\hat{\bbeta}^\top\bw \to^P a$
by \eqref{limit-a_*}, together with
$n^{-1} \sum_{i=1} U_i^2 = \Op(1)$ since
$\sum_{i=1} U_i^2\sim \chi^2_n$
we find
$$
\frac{1}{n}
\sum_{i\in I\cap \tilde I}
\|
\begin{pmatrix}
    \bx_i^\top\hat{\bbeta} - a U_i - \gamma\psi_i - \sigma G_i \\
    \bx_i^\top \tilde{\bbeta} - aU_i - \gamma\tilde \psi_i - \sigma\tilde G_i
\end{pmatrix}
\|^2
= \op(1) 
~~
\text{ where}
~~
\begin{pmatrix}
    G_i \\
    \tilde G_i
\end{pmatrix}
=
\begin{pmatrix}
    1 & \bar \eta \\
    \bar\eta & 1 \\
\end{pmatrix}^{1/2}
\bg_i.
$$
With probability approaching one, the second term in \eqref{reg}
is 0 for the large enough $K$ that we took at the beginning, and in this event the modified M-estimator 
$\hat{\bbeta}$ equals to the original M-estimator $\hat{\bb}(I)$ 
so that $\psi_i = - {\ell}_{y_i}(\bx_i^\top\hat{\bb})$
(cf. \Cref{lem:reg}), and similarly for $\tilde \bpsi$.
We have established
\begin{align*}
\frac{1}{n}\sum_{i\in I\cap\tilde{I}} \|\bx_i^\top\hat\bb + \gamma {\ell}_{y_i}'(\bx_i^\top\hat\bb) - a U_i - \sigma G_i
    \|^2 \equiv \frac{1}{n}\sum_{i\in I\cap\tilde{I}} \|-\text{Rem}_i\|_2^2 = \op(1).
\end{align*}
where we define $\text{Rem}_i$ by
$\bx_i^\top\hat\bb + \gamma {\ell}_{y_i}'(\bx_i^\top\hat\bb) =  a U_i + \sigma G_i + \text{Rem}_i$. Note that \( \bx_i^\top\hat\bb = \prox[\gamma{\ell}_{y_i}](a U_i + \sigma G_i + \text{Rem}_i) \) by definition of the proximal operator.
Now set \( \hat p_i = \prox[\gamma{\ell}_{y_i}](a U_i + \sigma G_i) \).
Because \( \prox[\gamma{\ell}_{y_i}](\cdot) \) is 1-Lipschitz,
\[
    \Bigl(
    \sum_{i\in I\cap \tilde I}
    \|\hat p_i - \bx_i^\top \hat\bb\|^2
    \Bigr)^{1/2}
    \le
    \Bigl(
        \sum_{i\in I\cap \tilde I}
    \|\text{Rem}_i\|^2
    \Bigr)^{1/2} = \op({\sqrt n}).
\]
Similarly, a proximal approximation holds for $\bx_i^\top \tilde{\bbeta}$
using $(U_i,\tilde G_i)$ instead.
We have to be a little careful here because $\bar \eta$
is independent of the $(G_i,\tilde G_i)$ but not of the $(U_i,y_i)$.
Using that $|{\ell}_{y_i}'|\le 1$, and that
$\bar \E[|A-B|] = \op(1)$ if $A,B$ are bounded random variables such that $|A-B|=\op(1)$, \eqref{ineq_eta_bar} gives
\begin{align*}
\bar \eta 
&= \frac{\delta \q^2\gamma^2}{\sigma^2}
\sum_{i\in I\cap \tilde I}
    \bar{\E}\Bigl[
\frac{ 
{\ell}_{y_i}'(\prox[\gamma{\ell}_{y_i}](a U_i+\sigma G_i))
{\ell}_{y_i}'(\prox[\gamma{\ell}_{y_i}](a U_i+\sigma \tilde G_i)
)
}{|I\cap \tilde I|}
\Bigr]
+ \op(1)
\end{align*}
where inside the conditional expectation $\bar\E[\cdot]$,
$(\bar\eta,U_i,y_i, I, \tilde{I})$ are fixed and integration is performed
with respect to the distribution of $(G_i,\tilde G_i)$. Thus, the above display can be rewritten  as 
\begin{align}\label{eq:random_eta_fixed_point}
    \bar\eta
    = \frac{\delta \q^2\gamma^2}{\sigma^2}
    \bar{\varphi}(\bar\eta) + \op(1)    
\end{align}
where $\bar{\varphi}:[-1,1]\to\R$ is the random function defined as
\begin{align*}
    \bar\varphi(t)
&=
\frac{1}{|I\cap\tilde I|}
\sum_{i\in I\cap \tilde I}
\bar\E\Bigl[
{\ell}_{y_i}'\Bigl(\prox[\gamma{\ell}_{y_i}](a U_i+\sigma G_i^t)\Bigr)
{\ell}_{y_i}'\Bigl(\prox[\gamma{\ell}_{y_i}](a U_i+\sigma \tilde G_i^t)\Bigr)
\Bigr]
\\
&=
\frac{1}{|I\cap\tilde I|}
\sum_{i\in I\cap \tilde I}
\iint
{\ell}_{y_i}'\Bigl(\prox[\gamma{\ell}_{y_i}](a U_i+\sigma g)\Bigr)
{\ell}_{y_i}'\Bigl(\prox[\gamma{\ell}_{y_i}](a U_i+\sigma \tilde g)\Bigr)
\phi_t(g,\tilde g)\mathrm{d}g\mathrm{d}\tilde g
\end{align*}
for all $t\in[-1,1]$, 
where $\phi_t:\R^2 \to (0, +\infty)$ is the density of two jointly centered normal
$(G^t,\tilde G^t)$ with $\E[(G^t)^2]=\E[(\tilde G^t)^2]=1$ and $\E[G^t\tilde G^t]=t$,
and in the first line $(G_i^t,\tilde G_i^t)\sim \phi_t$
is independent of $(\bar\eta,U_i,y_i)_{i\in[n]}$.
Notice that $\bar\varphi(t)$ can be viewed as an i.i.d. sum of random variables of $(y_i, U_i)$. Furthermore, since $I\cap\tilde I\subset [n]$ is independent of $(y_i, U_i)_{i\in [n]}$ and $|I\cap \tilde I|/n\to^p q^2(>0)$ by the property of hyper-geometric distribution (\Cref{rm:gepmetric_dist}), the weak law of large number implies the point-wise convergence:
$$
\forall t\in [-1,1], \quad 
\bar\varphi(t) \to^p \E\Bigl[
    {\ell}_{y}'\Bigl(\prox[\gamma{\ell}_{y}](a U+\sigma G^t)\Bigr)
    {\ell}_{y}'\Bigl(\prox[\gamma{\ell}_{y}](a U+\sigma \tilde G^t)\Bigr)
    \Bigr], 
$$
where $(y, U)=^d (y_i, U_i)$ and $(G^t, \tilde G^t)\sim \phi_t$. Taking $t=\eta$ for the deterministic solution $\eta$ 
of \eqref{eta_equation_robust} (with $a=0$ in robust regression)
or \eqref{eta_equation_logi} (in logistic regression),
 we get $\bar\varphi(\eta) \to^p  \sigma^2\eta/(\delta \q^2\gamma^2)$. Rearranging this result, we are left with 
\begin{align}\label{eq:deterministic_eta_fixed_point}
    \eta = \frac{\delta \q^2\gamma^2}{\sigma^2}\bar\varphi(\eta) +\op(1). 
\end{align}
Taking the difference between \eqref{eq:random_eta_fixed_point} and \eqref{eq:deterministic_eta_fixed_point}, 
using the mean-value theorem,
\begin{align}\label{eq:difference}
    \bar \eta-\eta
    =
    \frac{\delta \q^2\gamma^2}{\sigma^2} \Bigl(\bar\varphi(\bar\eta) - \bar\varphi(\eta)\Bigr) + \op(1)
    =
    \frac{\delta \q^2\gamma^2}{\sigma^2} \Bigl(\bar\eta-\eta\Bigr)
    \bar\varphi'(\bar{t})+ \op(1)    
\end{align}
for some (random) $\bar{t}$ between $\bar\eta$ and $\eta$.
By calculation similar to \eqref{varphi-d}-\eqref{F'}
thanks to \Cref{lm:varphi_derivative}, if $(G_i^t,\tilde G_i^t)$ has density $\phi_t$, with probability $1$, $\bar\varphi'(t)$ is non-negative for all $t\in[-1,1]$ and uniformly bounded from above as 
\begin{align*}
0\le \bar\varphi'(t)&=
\frac{ 1 }{|I\cap \tilde I|}
\frac{\sigma^2}{\gamma^2}
\sum_{i\in I\cap \tilde I}
\bar \E\Bigl[
\frac
{ \gamma {\ell}_{y_i}''(\prox[\gamma{\ell}_{y_i}](aU_i+\sigma G_i^t)) }
{ 1 +  \gamma {\ell}_{y_i}''(\prox[\gamma{\ell}_{y_i}](aU_i+\sigma G_i^t)) }
\frac
{\gamma  {\ell}_{y_i}''(\prox[\gamma{\ell}_{y_i}](aU_i+\sigma \tilde G_i^t))}
{1 + \gamma  {\ell}_i''(\prox[\gamma{\ell}_{y_i}](aU_i+\sigma \tilde G_i^t))}
\Bigr]\\
&\le
\frac{ 1 }{|I\cap \tilde I|}
\frac{\sigma^2}{\gamma^2}
\sum_{i\in I\cap \tilde I}
\bar \E\Bigl[
\frac
{ \gamma {\ell}_{y_i}''(\prox[\gamma{\ell}_{y_i}](aU_i+\sigma G_i^t)) }
{ 1 +  \gamma {\ell}_{y_i}''(\prox[\gamma{\ell}_{y_i}](aU_i+\sigma G_i^t)) }
\Bigr]
\\
&=
\frac{ 1 }{|I\cap \tilde I|}
\frac{\sigma^2}{\gamma^2}
\sum_{i\in I\cap \tilde I}
\int\Bigl[
\frac
{ \gamma {\ell}_{y_i}''(\prox[\gamma{\ell}_{y_i}](aU_i+\sigma g)) }
{ 1 +  \gamma {\ell}_{y_i}''(\prox[\gamma{\ell}_{y_i}](aU_i+\sigma g)) }
\Bigr]\frac{e^{-g^2/2}}{\sqrt{2\pi}}\mathrm{d}g
\quad\text{since } G_i^t\sim N(0,1). 
\end{align*}
Note that the RHS is independent of $t\in[-1,1]$. Furthermore, by the same argument we used to derive the limit of $\bar\varphi$ above, the law of large numbers and the nonlinear system (equation \eqref{Robust2} in robust regression and equation \eqref{3} in logistic regression) imply $\text{RHS}\to^p \sigma^2/(\gamma^2 q\delta)$. Putting this result and the above inequality of $\bar\varphi'(t)$ with $t=\bar{t}$ together, we get the following estimate of $\bar{\varphi}'(\bar{t})$:
$$0 \le \bar{\varphi}'(\bar{t}) \le  \sigma^2/(\gamma^2 q\delta) + \op(1). 
$$
Combining this result and \eqref{eq:difference}, we are left with 
$$
|\bar\eta - \eta|
=
|\bar \eta - \eta| \frac{\delta \q^2\gamma^2}{\sigma^2} |\bar\varphi'(\bar{t})| + \op(1)
\le
|\bar \eta - \eta| \frac{\delta \q^2\gamma^2}{\sigma^2} \frac{\sigma^2}{\gamma^2 \q\delta}
+ \op(1) = \q|\bar{\eta}-\eta| + \op(1)
$$
and $\bar\eta - \eta = \op(1)$ thanks to $\q\in (0,1)$. 
Since $\bar\eta = \hat{\bbeta}^\top\bP\tilde{\bbeta}/\sigma^2 + \op(1)$
by \eqref{concentration-eta-bar}, the proof 
of \eqref{eq:thm_eta_robust} and \eqref{eq:thm_eta_logi}
is complete.
Next, \eqref{eq:thm_estimation}
follows from 
\Cref{lem:SOS} and \Cref{lm:relation_trV_trA_gamma}.

Finally for \eqref{eq:thm_distribution_robust}
and \eqref{eq:thm_eta_distribution_logi},
by symmetry
$\E[\text{LHS}_i \mid I, \tilde I]$ is the same for all $i\in I\cap \tilde I$.
In particular, the maximum of the conditional expectation is the same as 
the average over $I\cap \tilde I$, so that
$\sum_{i\in I\cap \tilde I}\E[\text{LHS}_i \mid I, \tilde I] \le C$
proved above gives
$\max_{i\in I\cap \tilde I}
\E[\text{LHS}_i \mid I, \tilde I] = \Op(1/n)$
since $I\cap \tilde I$ has cardinality of order $n$.
Finally, we have
\begin{equation}\bW^\top\bW\to^P
\sigma^2
\begin{pmatrix}
    1 & \eta
\\  \eta & 1
\end{pmatrix},
\qquad
\Bigl(\bW^\top\bW\Bigr)^{1/2}\to^P
\sigma
\begin{pmatrix}
    1 & \eta
\\  \eta & 1
\end{pmatrix}^{1/2},
\label{cont_mapping_matrix_sq}
\end{equation}
by continuity of the matrix square root and the continuous
mapping theorem (or, alternatively, by reusing the argument in \eqref{39}).
Using again $\tr[\bP\bA]\to^P\gamma$, $\hbbeta^T\bw\to^Pa$,
$(\hat{\bbeta}^\top \bP\bA\bX^\top \bD)\be_i\to^P 0$,
and similarly for $\tbbeta$, combined with \eqref{cont_mapping_matrix_sq},
we obtain \eqref{eq:thm_distribution_robust} 
and \eqref{eq:thm_eta_distribution_logi}.

\section{Auxiliary lemmas}
\label{sec:auxilliary_lemma}

\subsection{Approximate multivariate normality}\label{subsec:multivariate_normality}

\begin{proposition}
    \label{lem:W}
    Let $\bz\sim N(\bm{0}_p, \bI_p)$ and 
    let $\bW:\R^n\to\R^{p\times M}$ be a locally Lipschitz function with $M\le p$. 
    Then there exists $\bg\sim N(\bm{0}_M, \bI_M)$ such that
    $$
    \E\Bigl[\bigl\|\bW(\bz)^\top \bz - \sum_{j=1}^p  \frac{\partial \bW(\bz)^\top \be_j}{\partial z_j} - \Bigl\{\bW(\bz)^\top \bW(\bz)\Bigr\}^{1/2} \bg
    \bigr\|^2\Bigr]\le \C \sum_{j=1}^p \E\Bigl[\bigl\| \frac{\partial \bW(z)}{\partial z_j}\bigr\|_{\F}^2\Bigr], 
    $$
    where $\{\cdot\}^{1/2}$ is the square root of the positive semi-definite matrix. 
\end{proposition}
This moment inequality is a matrix-generalization of  \cite[Proposition 13]{bellec2022derivatives} and \cite[Theorem 2.2]{bellec2023debiasing}. 
It is particularly useful to show that as $p\to+\infty$ with fixed $M$, and provided that $\sum_{j=1}^p \E[\| (\partial \bW(z)/\partial z_j)\|_{\F}^2]$ is suitably bounded, 
the following random vector (which is mean-zero by Stein's lemma)
$$
\bW(\bz)^\top \bz - \sum_{j=1}^p  \frac{\partial \bW(\bz)^\top \be_j}{\partial z_j} \in \R^M
$$
is approximately multivariate normal (in the $L_2$ sense) with covariance
approximated by
$\bW(\bz)^\top\bW(\bz)$. In our paper, as shown in \Cref{sec:proof}, we apply this inequality with $W = [P\hat\bbeta, P\tilde{\bbeta}]\in \R^{p\times 2}$, using the derivative formula \eqref{eq:derivative_formula} in \Cref{lm:derivative_formula}. 

\begin{proof}
    Let $\tilde \bz$ be an independent copy of $\bz$ and let $\tilde \bW = \bW(\tilde \bz)$. Noting $M\le p$, we denote the SVD of $\tilde\bW\in\R^{p\times M}$ by
    $\tilde \bW = \sum_{m=1}^M s_m \bu_m \bv_m^\top$ where 
    $s_1\ge s_2\ge ... \ge s_M\ge 0$ are the singular values. Here, we allow some $s_m$ to be $0$ to have $M$ terms by adding extra terms if necessary,
    so that $(\bv_1,...,\bv_M)$ is an orthonormal basis in $\R^M$. Now we define $\tilde{\bQ} = \sum_{m=1}^M \bv_{m} \bu_{m}^\top\in \R^{M\times p}$,  
    so that $\tilde{\bW} = \tilde{\bQ}^\top (\tilde \bW^\top \tilde \bW)^{1/2}$ thanks to $(\tilde \bW^\top \tilde \bW)^{1/2} = \sum_{m=1}^M s_m \bv_m\bv_m^\top$. 
    Define $\bg = \tilde{\bQ} \bz$ 
    and note that$\bg\sim N(\bm{0}_M,\bI_M)$ since $\tilde \bW =\bW(\tilde \bz)$ is independent of $\bz$. 
    With $\bW=\bW(\bz)$ (omitting the dependence in $\bz$), using $\bg = \tilde\bQ \bz$ and $\tilde{\bW} = \tilde{\bQ}^\top (\tilde \bW^\top \tilde \bW)^{1/2}$, we have
    $$
    \tilde \bW^\top \bz - (\bW^\top\bW)^{1/2}\bg
    =
    \bigl[(\tilde \bW^\top \tilde \bW)^{1/2}  - (\bW^\top\bW)^{1/2}\bigr]\bg.
    $$
    Applying the Second order Stein formula in 
    \cite{bellec2021second}
    (see also 5.1.13 in \cite{bogachev1998gaussian})
    to $\bm{U}(\bz)=\bW(z)^\top-\{\bW(z)^\top \bW(\bz)\}^{1/2}\tilde \bQ\in\R^{M\times p}$
    conditionally on $(\tilde \bz,\tilde \bQ)$, 
    we find
    \begin{align}
    \E\Bigl[\|\bW^\top\bz - \sum_{j=1}^p \frac{\partial (\bW^\top-\{\bW^\top\bW\}^{1/2}\tilde \bQ)\be_j}{\partial z_j} - \{\bW^\top\bW\}^{1/2}\bg \|^2\Bigr] 
    &= \E\bigl[\|\bm{U} \bm{z} - \sum_{j=1}^p \frac{\partial\bm{U}\be_j}{\partial z_j} \|^2\bigr]
    \nonumber \text{ by $\bg = \tilde{\bQ}\bz$}
  \\&\le
  \E\Bigl[\|\bm{U}(\bz)\|_{\F}^2
  +\sum_{j=1}^p \|\frac{\partial \bm{U}(\bz)}{\partial z_j}\|_{\F}^2
    \Bigr].
  \label{SOS}
    \end{align}
    Since $\tilde \bW^\top = (\tilde \bW^\top \tilde \bW)^{1/2}\tilde \bQ$,
    by the triangle inequality for $\|\bm{U}\|_{\F} =
  \|\bW^\top - \{\bW^\top\bW\}^{1/2}\tilde\bQ\|_{\F}$, 
    \begin{align}
        \|\bm{U}\|_{\F} 
  &\le \|\bW-\tilde \bW\|_{\F} + \|\tilde \bW^\top  - (\bW^\top\bW)^{1/2}\tilde \bQ\|_{\F} \nonumber\\
  &=
  \|\bW-\tilde \bW\|_{\F} + \|[(\tilde \bW^\top \tilde \bW) - (\bW^\top\bW)^{1/2}]\tilde \bQ\|_{\F}
  \nonumber\\
    &\le
  \|\bW-\tilde \bW\|_{\F} + \sqrt 2\|\bW-\tilde\bW\|_{\F}
  \label{piece1}
    \end{align}
    thanks to $\|\tilde \bQ\|_{\oper}\le 1$ and using, for the last line,
    inequality
    \begin{equation}
    \|(\tilde \bW^\top\tilde \bW)^{1/2}-(\bW^\top\bW)^{1/2}\|_{\F}\le \sqrt 2 \|\bW-\tilde \bW\|_{\F} 
    \label{lip}
    \end{equation}
    from \cite{araki1981inequality,chun1989perturbation}.
    Now for the second term in \eqref{SOS}, using the inequality $(a+b)^2 \le 2(a^2 + b^2)$ for $  \sum_{j=1}^p \|(\partial \bm{U})/(\partial z_j)\|_{\F}^2
         = \sum_{j=1}^p \| (\partial/\partial z_j) (\bW^\top - (\bW^\top\bW)^{1/2}\tilde \bQ)\|_{\F}^2$, 
    \begin{align}
  \sum_{j=1}^p \|\frac{\partial \bm{U}}{\partial z_j}\|_{\F}^2 \le
          2\sum_{j=1}^p
          \| \frac{\partial \bW}{\partial z_j}\|_{\F}^2
          +
          \| \frac{\partial (\bW^\top\bW)^{1/2}}{\partial z_j}\tilde \bQ\|_{\F}^2
          \le
          4\sum_{j=1}^p
          \| \frac{\partial \bW}{\partial z_j}\|_{\F}^2
        \label{piece2}
    \end{align}
    where for the last line we used again inequality \eqref{lip} valid
    for any two $\tilde \bW,\bW$, which grants
    \begin{equation}
    \|\frac{\partial (\bW^\top\bW)^{1/2}}{\partial z_j}\|_{\F}
    \le \sqrt{2} \|\frac{\partial \bW}{\partial z_j}\|_{\F}
    \label{eq:lip-derivative}
    \end{equation}
    by definition of the directional
    derivative and continuity of the Frobenius norm.

    It remains to bound from above the divergence term appearing 
    in the left-hand side of \eqref{SOS}.
    For each $m\in[M]$,
    $
    \be_m^\top \sum_{j=1}^p  (\partial/\partial z_j)((\bW^\top\bW)^{1/2}) \cdot  \tilde \bQ \be_j
    $
    is the divergence of the vector field
    $\R^p \ni \bz\mapsto \tilde{\bQ}^\top (\bW^\top\bW)^{1/2}\be_m \in\R^p$.
    Since $\tilde \bQ\in\R^{M\times p}$ is fixed and its rank is at most $M$,
    the Jacobian of this vector field is of rank $M$ at most.
    Thus, the divergence (trace of the Jacobian) is smaller
    than $\sqrt M$ times the Frobenius norm of the Jacobian.
    This gives for every $m\in [M]$ the following bound on the square
    of the divergence:
    $$
    |\be_m^\top \sum_{j=1}^p \frac{\partial(\bW^\top\bW)^{1/2}}{\partial z_j} \tilde \bQ \be_i|^2
    \le M \sum_{j=1}^p \|\be_m^\top \frac{\partial (\bW^\top\bW)^{1/2}\tilde \bQ}{\partial z_j}\|^2.
    $$
    Summing over $m\in[M]$ we find
    \begin{equation}
        \label{piece3}
    \|\sum_{j=1}^p \frac{\partial(\bW^\top\bW)^{1/2}}{\partial z_j} \tilde \bQ \be_j
    \|^2
    \le M\sum_{j=1}^p \|
    \frac{\partial (\bW^\top\bW)^{1/2}\tilde \bQ}{\partial z_j}
        \|_{\F}^2.
    \end{equation}
    Since $\|\tilde \bQ\|_{\oper}\le 1$, we can further upper-bound by removing $\tilde \bQ$ inside the Frobenius norm, and use again \eqref{eq:lip-derivative}.
    Combining the pieces \eqref{SOS}, \eqref{piece1}, \eqref{piece2}, \eqref{piece3}, we find
    \begin{align*}
    \E\Bigl[\|\bW^\top\bz - \sum_{j=1}^p \frac{\partial \bW^\top\be_j}{\partial z_j} - (\bW^\top\bW)^{1/2}\bg \|^2\Bigr]
  \le
  \C
  \E\Bigl[\|\bW-\tilde \bW\|_{\F}^2
  +\sum_{j=1}^p \|\frac{\partial \bW}{\partial z_j}\|_{\F}^2
    \Bigr].
\end{align*}
    Since $\bW,\tilde \bW$ are iid, using the triangle inequality for the Frobenius norm  with $(a+b)^2 \le 2(a^2+b^2)$ and 
     the Gaussian Poincar\'e inequality finally yield 
    $
    \E[\|\bW-\tilde \bW\|_{\F}^2]\le 4\E[\|\bW-\E[{\bW}]\|_{\F}^2] \le C [
        \sum_{i=1}^n \|(\partial/\partial) z_i \bW\|_{\F}^2]
    $
        and the proof is complete.
\end{proof}
\subsection{Derivative of $F(\eta)$}
\begin{lemma}\label{lm:varphi_derivative}
Let $G$ and $Z$ be independent $N(0,1)$ random variables. 
Then for any Lipschitz continuous
function $f$ with $\E[f(G)^2]<+\infty$, the map 
$\varphi: \eta \in [-1,1] \mapsto \E[f(G)f(\eta G + \sqrt{1-\eta^2}Z)]\in\R$ has derivative
$\varphi'(\eta) = \E[f'(G)f'(\eta G + \sqrt{1-\eta^2}Z)]$. 
\end{lemma}
\begin{proof}
Since $f$ is Lipschitz and $N(0,1)$ has no point mass, $f$ is differentiable at $G\sim N(0,1)$ with probability $1$, so by the dominated convergence theorem, we have 
$$\varphi'(\eta) = \E\Bigl[
    f(G) f'\bigl(\eta G +\sqrt{1-\eta^2}Z\bigr) 
    \bigl(G-\frac{\eta}{\sqrt{1-\eta^2}} Z\bigr) 
\Bigr] =  \frac{1}{\sqrt{1-\eta^2}} \E\Bigl[f\bigl(\eta A + \sqrt{1-\eta^2}B\bigr) f'(A) B\Bigr],
$$
where we defined $A=\eta G + \sqrt{1-\eta^2}Z$ and $B=\sqrt{1-\eta^2}G-\eta Z$ so that $(A, B)$ are again independent with $A, B=^d N(0,1)$. Using Stein's formula for $B$ conditionally on $A$, noting that $(1-\eta^2)^{1/2}$ is cancelled out, we complete the proof. 
\end{proof}
\subsection{Modified loss and moment
inequalities}\label{subsec:modified_loss}
This subsection provides useful approximations to study
two estimators $\hbb(I),\hbb(\tilde I)$ trained on two subsampled
datasets indexed in $I$ and $\tilde I$. These approximations are used
in the proof of the main result in \Cref{sec:proof},
with the key ingredient being \Cref{lem:SOS}.
The approximations in this subsection are obtained as a consequence
of the moment inequalities given 
in \Cref{lm:chi_square,lm:second_order_stein} developed in
\cite{bellec2020out} for estimating the out-of-sample error
of a single estimator.
Because the moment inequalities in \Cref{lm:chi_square,lm:second_order_stein}
requires us to bound from above expectations involving $\hbb(I),\hbb(\tilde I)$
and their derivatives, we resort to the following modification
of the M-estimators (introduced in \cite[Appendix D.4]{bellec2022observable})
to guarantee that any finite moment of
$\hbb(I),\hbb(\tilde I)$
and their derivatives are suitably bounded.

\begin{lemma}
    \label{lem:reg}
    Let $\hat{\bb}(I) \in \argmin_{\bb\in\R^p} \sum_{i\in I} {\ell}_{y_i} (\bx_i^\top\bb)$ be the M-estimator fitted on the subsampled data $(\bx_i, y_i)_{i\in I}$. 
    Now, for any positive constant $K>0$ and any twice continuous differentiable function $H:\R\to\R$ such that $H'(u)=0$ for $u\le 0$ and $H'(u)=1$ for $u\ge 1$, we define the modified M-estimator $\hat{\bbeta}(I)$ as 
    \begin{equation}
        \label{reg}
    \hat{\bbeta}(I) \in \argmin_{\bbeta\in\R^p}  \mL(\bX{\bbeta}) \quad \text{where} \quad \mL(\bu) = \sum_{i\in I} {\ell}_{y_i}(u_i) + |I| H\Bigl(\frac{1}{2|I|}\sum_{i\in I} u_i^2 - \frac{K}{2}\Bigr)
    \end{equation}
    for $\bu\in\R^n$.
    If the vanilla M-estimator $\hat{\bb}(I)$  exists with high probability and 
    $\PP(\|\bX\hat{\bb}(I)\|^2/n \le K )\to 1$ holds for a sufficiently large $K>0$,
    then on the event $\{\|\bX\hat{\bb}(I)\|^2/n \le K \}$ the vanilla and modified M-estimators coincide, i.e.,  $\hat{\bb}(I)=\hat{\bbeta}(I)$. 
\end{lemma}

\begin{lemma}\label{lm:derivative_formula}
Assume that ${\ell}_{y_i}$ is twice-continuously differentiable
with ${\ell}_{y_i}''(u) \vee |{\ell}_i'(u)| \le 1$ and ${\ell}_{y_i}''(u)>0$ for all $u\in\R$.
Fix any $K>0$ and let $\hat{\bbeta}$ be the M-estimator with the modified loss
\eqref{reg}
and let $\bpsi = -\nabla\mL(\bX\hat{\bbeta})$. Then, the maps $\bX\in\R^{n\times p}\mapsto \hat{\bbeta}(\by,\bX)\in\R^p$ and $\bX\in\R^{n\times p}\mapsto \bpsi(\by,\bX)\in \R^n$
are continuously differentiable, with its derivatives given by 
\begin{align}\label{eq:derivative_formula}
    \frac{\partial \hat{\bbeta}}{\partial x_{ij}} = \bA (\be_j \psi_i-\bX^\top \bD \be_i \hat\beta_j), \quad 
    \frac{\partial \bpsi }{\partial x_{ij}} = -\bD \bX \bA \be_j \hat\psi_i - \bV \be_i \hat\beta_j
\end{align}
for all $i\in[n], j\in [p]$, 
where $\bD = \nabla^2 \mL(\bX\hat{\bbeta})\in\R^{n\times n}$, $\bA=(\bX^\top \bD \bX)^{-1}\in\R^{p\times p}$, $\bV=\bD-\bD \bX\bA\bX^\top \bD\in \R^{n\times n}$. 
Here,  $\sum_{i\in I}\|\bx_i^\top\hat\bbeta\|^2$, $\|\bpsi\|^2$ and $\|\bD\|_{\oper}$ are bounded from above as
\begin{equation}
\sum_{i\in I}(\bx_i^\top\hat\bbeta)^2 \le |I| (K+2), \quad \|\bpsi\|^2 \le |I| (1+\sqrt{K+2})^2, \quad \|\bD\|_{\oper} \le C(K, \q, \delta)
\label{eq:as_bound}
\end{equation}
with probability $1$ 
and $\bm{0}_{n\times n} \preceq \bV \preceq \bD$.
Finally, we have
for all integer $m\ge 1$
\begin{equation}
\E[\|\hat{\bbeta}\|^m] \vee \E[\|n\bA\|_{\oper}^m] \le
\begin{cases}
    C(m, K, \q, \delta, \rho, \mathsf{Law}(\eps_i)) & \text{under \Cref{assum:robust}}, \\
    C(m, K, \q, \delta) & \text{under \Cref{assum:logi}}.
\end{cases}
\label{eq:moment_bound}
\end{equation}
\end{lemma}
\begin{proof}
    The proof of the first part of the lemma 
    and \eqref{eq:as_bound}
    is given in
    Appendix D.4 in \cite{bellec2022observable}.
    The moment bound \eqref{eq:moment_bound} is proved in 
    \cite[Appendix D.4]{bellec2022observable}
    under \Cref{assum:logi} when $y_i$ is binary valued.
    We now prove \eqref{eq:moment_bound} under \Cref{assum:robust}.
    Let also $\bV,\bA$ be the matrices
    defined in \Cref{lm:derivative_formula} for $\hbbeta$,
    and let
    $\tilde\bV,\tilde \bA$ be corresponding matrices
    defined in \Cref{lm:derivative_formula} for $\tbbeta$.

    By \eqref{eq:as_bound}, we have 
    $\|\hbbeta\|^2 \le \|(|I|^{-1}\sum_{i\in I}\bx_i\bx_i^T)^{-1}\|_{\oper}
    (K+2)$
    so that the bound on $\E[\|\hbbeta\|^m]$ follows by the known result
    $\E[\|(|I|^{-1} \sum_{i\in I}\bx_i\bx_i^T)^{-1}\|_{\oper}^m]
    \le C(\delta\q, m)$ which follows from the integrability
    of the density of the smallest eigenvalue of a Wishart matrix
    (\cite{edelman1988eigenvalues}), as explained for instance in
    \cite[Proposition A.1]{bellec2023debiasing}.

    Let $\alpha>0$ be a constant such that $1-\alpha > (\delta\q)^{-1}$
    and let $Q_\alpha\in\R$ be the quantile such that
    $\PP(|\eps_i|\le Q_\alpha)=1-\alpha/2$.
    Since $\q\delta>1$ and $|I|=(\delta\q) p$,
    by the weak law of large numbers applied to the indicator functions
    $\mathbb I\{|\eps_i|\le Q_\alpha\}$,
    with probability approaching one,
    there exists a random set
    $\hat I\subset I$ with $p(\delta\q)(1-\alpha)=|I|(1-\alpha)\le |\hat I|$
    and 
    $\sup_{i\in \hat I}|\eps_i|\le Q_\alpha$.
    Next, by \eqref{eq:as_bound}, there exists a constant $C(\delta,\q,\alpha,K)$ such that
    $|\hat I|^{-1}\sum_{i\in \hat I} 
    (\bx_i^T\hbbeta)^2 \le C(\delta,\q,\alpha, K)$.
    Now define
    $$\check I =\Bigl\{ i\in \hat I: (\bx_i^T\hbbeta)^2
    \le
    \frac{C(\delta,\q,\alpha,K)}{1-\sqrt{\delta\q(1-\alpha)}}
    \Bigr\}
    $$
    and note that by Markov's inequality,
    $
    |\hat I \setminus \check I| / |\hat I|
    \le (1-\sqrt{\delta\q(1-\alpha)})
    $.
    This gives
    $|\check I|\ge\sqrt{\delta\q(1-\alpha)}|\hat I|
    \ge p (\delta\q(1-\alpha))^{3/2}$ and 
    the constant $(\delta\q(1-\alpha))^{3/2}$ is strictly larger than 1.
    Finally, since for all $i\in \check I$ we have
    $|\eps_i|\le Q_\alpha$ and
    $(\bx_i^T\hbbeta)^2\le C(\delta,\q,\alpha,K)/(1-\sqrt{\delta\q(1-\alpha)}$,
    for all $i\in \check I$ we have $\eps_i-\bx_i^T\hbbeta \in [-L,L]$
    for some constant $L=L(\delta,\q,\alpha, K,Q_\alpha)$.
    Finally,
    \begin{align*}
        \|n \bA\|_{\oper}
        &\le 
        \frac{n}{|\check I|}
        \Bigl\|\Bigl(\frac{1}{|\check I|}\sum_{i\in \check I}\bx_i \rho''(y_i-\bx_i^T\hbbeta) \bx_i^T\Bigr)^{-1}\Bigr\|_{\oper} \le
        \frac{\delta
        \max_{u\in[-L,L]}(\rho''(u)^{-1})
        }{(\delta\q(1-\alpha))^{3/2}}
        \Bigl\|\Bigl(\frac{1}{|\check I|}\sum_{i\in \check I}\bx_i \bx_i^T\Bigr)^{-1}\Bigr\|_{\oper}.
    \end{align*}
    Since $\rho''$ is positive and continuous, the moment of order $m$
    of the previous display is bounded from above
    by some $C(m,\delta, K,q,Q_\alpha,\rho)$ thanks to
    the explicit formula of \cite{edelman1988eigenvalues}
    for the density of the smallest eigenvalue of a Wishart matrix,
    as explained in
    \cite[Lemma D.2]{bellec2022observable}.
\end{proof}

\newcommand\assumLemma{
    Under the assumptions and notation in \Cref{lem:SOS}, we have
}
\begin{lemma}\label{lem:SOS}
    Let either \Cref{assum:robust} or \Cref{assum:logi}
    be fulfilled with $I,\tilde I$ independent and uniformly distributed
    over all subsets of $[n]$ of size $\q n$.
    Let the notation of \Cref{sec:proof}
    be in force for $(\hbbeta,\bpsi,\bA,\bV)$
    (as in \Cref{lem:reg,lm:derivative_formula} for $I$)
    and similarly for
    $(\tbbeta,\tbpsi,\tilde\bA,\tilde\bV)$.
    Then, 
    \begin{equation*}
    \tr[\bV]\cdot \hat{\bbeta}^\top \bP \tilde{\bbeta} - \tr[\bP \tilde\bA] \cdot \bpsi^\top \tilde\bpsi = \Op(n^{1/2}).
    \end{equation*}
\end{lemma}

\begin{proof}
    We will apply \Cref{lm:second_order_stein} below with $\bm{\rho}={\bpsi}$ and $\bm{\eta}=\bP\tilde{\bbeta}$. 
    \begin{lemma}[Proposition 2.5 in \cite{bellec2020out}]\label{lm:second_order_stein}
    Let $\bX = (x_{ij})\in\R^{n\times p}$ with iid $N(0,1)$ entries and $\bm{\rho}:\R^{n\times p}\to\R^n$, $\bm{\eta}:\R^{n\times p}\to\R^p$ be two vector functions, with weakly differentiable components $\rho_1, \dots, \rho_n$ and $\eta_1, \dots, \eta_p$. Then
    \begin{multline*}
    \E\Bigl[
        \Bigl(\bm{\rho}^\top \bm{X} \bm{\eta} - \sum_{i=1}^n\sum_{j=1}^p \frac{\partial (\rho_i\eta_j)}{\partial x_{ij}}
        \Bigr)^2
    \Bigr] 
    \le \E\bigl[\|\bm{\rho}\|^2 \|\bm{\eta}\|^2\bigr] + 2\E\Bigl[\sum_{i=1}^n\sum_{j=1}^p \|\bm{\eta}\|^2\|\frac{\partial \bm{\rho}}{\partial x_{ij}}\|^2 + \|\bm{\rho}\|^2\|\frac{\partial \bm{\eta}}{\partial x_{ij}}\|^2 \Bigr].
    \end{multline*}
\end{lemma}
    Using the derivative formula \eqref{eq:derivative_formula} and upper bounds \eqref{eq:as_bound} in \Cref{lm:derivative_formula}, it holds that 
    \begin{align*}
        \sum_{i=1}^n\sum_{j=1}^p \|\frac{\partial {\bpsi}}{\partial x_{ij}}\|^2
        &\le 2 \| \bD \bX\bA\|_{\F}^2 \|{\bpsi}\|^2 + 2 \|\bV\|_{\F}^2 \|\hat{\bbeta}\|^2
        \le \C (n^2 \|\bX\|_{\oper}^2 \|\bA\|_{\oper}^2 + n \|\hat{\bbeta}\|^2), \\
        \sum_{i=1}^n\sum_{j=1}^p  \|\frac{\partial \bP \tilde{\bbeta}}{\partial x_{ij} }\|^2 
        &\le 2 \|\bP \tilde \bA\|_{\F}^2 \|\tilde\bpsi\|^2 + 2 \|\bP\tilde\bA\bX^\top  \tilde \bD \|_{\F}^2 \|\hat{\bbeta}\|^2
        \le \C(pn\|\tilde\bA\|_{\oper}^2 + n \|\bX\|_{\oper}^2 \|\tilde\bA\|_{\oper}^2 \|\tilde{\bbeta}\|^2).
    \end{align*}
    Since
    $\E[\|\bX {n^{-1/2}}\|_{\oper}^k] \vee \E[\|{n}\bA\|_{\oper}^2] \vee \E[\|\hat{\bbeta}\|^k]\le C$
    for a constant independent of $n,p$
    by the moment bounds \eqref{eq:moment_bound}
    and integration of $\PP(\|\bX {n^{-1/2}}\|_{\oper}>1+\delta^{-1/2}+tn^{-1/2})\le e^{-t^2/2}$ (see, e.g.,
    \cite[Theorem II.13]{DavidsonS01},
    \cite[Theorem 7.3.1]{vershynin2018high}
    or
    \cite[Theorem 5.5]{boucheron2013concentration}),
    we obtain
    since $\bpsi=^d\tbpsi$ and $\bbeta=^d\tbbeta$,
    \begin{equation}
        \sum_{i=1}^n\sum_{j=1}^p 
        \E\Bigl[
        \|\frac{\partial \bP {\hbbeta}}{\partial x_{ij} }\|^2 
        +
        \|\frac{\partial \bP \tilde{\bbeta}}{\partial x_{ij} }\|^2 
        +
        \frac1n \|\frac{\partial {\bpsi}}{\partial x_{ij}}\|^2  
        +
        \frac1n \|\frac{\partial {\bpsi}}{\partial x_{ij}}\|^2 
        \Bigr]
        \le C'
        \label{bound_frobenius_norm_derivatives}
    \end{equation}
    for another constant independent of $n,p$.
    Thus the RHS of \Cref{lm:second_order_stein} is $O(n)$. This gives 
     $$
     \E\Bigl[\Bigl({\bpsi}^\top \bX \bP\tilde{\bbeta} - \sum_{i=1}^n\sum_{j=1}^p \frac{\partial }{\partial x_{ij}} ({\be_i^\top \bpsi\cdot\be_j^\top \bP\tilde{\bbeta}})\Bigr)^2 \Bigr] \le n \C.
     $$
  Using the formula \eqref{eq:derivative_formula} again, 
     \begin{align*}
        \sum_{i=1}^n\sum_{j=1}^p \frac{\partial }{\partial x_{ij}} ({\be_i^\top \bpsi \be_j^\top \bP\tilde{\bbeta}}) &= \sum_{ij} \be_i^\top  \bigl(\frac{\partial \bpsi}{\partial x_{ij}}\bigr)  \be_j^\top \bP\tilde{\bbeta} + \be_i^\top \bpsi \be_j^\top\bP \bigl(\frac{\partial \tilde{\bbeta}}{\partial x_{ij}}\bigr)\\
        &= -\bpsi^\top \bD \bX\bA \bP\tilde{\bbeta}-\tr[\bV]\hat{\bbeta}^\top \bP\tilde{\bbeta}
        + \tr[{\tilde{\bA}}]\bpsi^\top \tilde{\bpsi} - \tilde{\bbeta}^\top \bP {\tilde{\bA}} \bX^\top \tilde{\bD}\bpsi. 
     \end{align*}
    Using the almost sure bounds \eqref{eq:as_bound} and the moment bounds \eqref{eq:moment_bound}, 
     \begin{align*}
        \E[|\bpsi^\top \bD \bX\bA \bP\tilde{\bbeta}|^2] &\le \E[\|\bpsi\|^2\|\tilde{\bbeta}\|^2 \|\bD \bX\bA\bP\|_{\oper}^2] \le \C \E[n \|\tilde{\bbeta}\|^2 \|\bX\bA\|_{\oper}^2] = O(1)\\
        \E[| \tilde{\bbeta}^\top \bP {\tilde{\bA}} \bX^\top \tilde{\bD}\bpsi|^2] &\le \E[\|\bpsi\|^2\|\tilde{\bbeta}\|^2 \|\bP {\tilde{\bA}} \bX^\top \tilde{\bD}\|_{\oper}^2] \le \E[n \|\tilde{\bbeta}\|^2 \|{\tilde{\bA}} \bX^\top\|_{\oper}^2] = O(1).
     \end{align*}
     This gives
     $$
     \E\Bigl[\Bigl({\bpsi}^\top \bX \bP\tilde{\bbeta} + \tr[\bV]\hat{\bbeta}^\top \bP\tilde{\bbeta} - \tr[\bP {\tilde{\bA}}]\bpsi^\top \tilde{\bpsi}\Bigr)^2\Bigr] = O(n) + O(1).
     $$
     Here, ${\bpsi}^\top\bX \bP\tilde{\bbeta} $ is $0$ by the KTT condition $\bX^\top \bpsi = \bm{0}_p$, and the proof is complete. 
\end{proof}

\begin{lemma}
    \assumLemma
\begin{equation}
    \label{eq:relation_bpsi_P_beta}
\|\bpsi\|^2 - p^{-1} \tr[\bV]^2 \|\bP\hat\bbeta\|^2
= \Op(n^{1/2}). 
\end{equation}
\end{lemma}
\begin{proof}
    We will use \Cref{lm:chi_square} below  with $\bm{\rho}=\bpsi/(\sqrt{n\q} (1+\sqrt{K+2}))$. 
    \begin{lemma}[Theorem 2.6 in \cite{bellec2020out}]\label{lm:chi_square}
    Assume that $\bX=(x_{ij})\in\R^{n\times p}$ has iid $N(0,1)$ entries, that $\bm{\rho}:\R^{n\times p}\to \R^n$ is weakly differentiable and that $\|\bm{\rho}\|^2 \le 1$ almost everywhere. Then
    \begin{align*}
      \E\Bigl|p\|\bm{\rho}\|^2 
      -\sum_{j=1}^p \Bigl(
          \bm{\rho}^\top \bX \be_j - \sum_{i=1}^n \frac{\partial \rho_i}{\partial x_{ij}}
      \Bigr)^2
      \Bigr|  \le C \E\Bigl[1+\sum_{i=1}^n\sum_{j=1}^p \|\frac{\partial \bm{\rho}}{\partial x_{ij}}\|^2\Bigr]^{1/2} \sqrt{p} + C \E\Bigl[\sum_{i=1}^n\sum_{j=1}^p \|\frac{\partial \bm{\rho}}{\partial x_{ij}}\|^2\Bigr],
    \end{align*}
    where $C>0$ is an absolute constant. 
\end{lemma}
    
    Note $\|\bpsi\|^2 \le n\q (1+\sqrt{K+2})^2$ with probability $1$ from the almost sure bound \eqref{eq:as_bound} in \Cref{lm:derivative_formula}, so the assumption in \Cref{lm:chi_square} is satisfied.
    In logistic regression, we can assume by rotational invariance
    that $\bbeta^*/\|\bbeta^*\|=\be_1$ (first canonical basis vector),
    and we apply \Cref{lm:chi_square} conditionally on $(\by,\bX\bbeta^*)$
    to the Gaussian matrix $(x_{ij})_{i\in[n], j\ge 2}$.
    In robust regression, we apply \Cref{lm:chi_square}
    with respect to the full Gaussian
    matrix $\bX = (x_{ij})_{i\in[n], j\ge 2}$, conditionally
    on the independent noise $(\eps_i)_{i\in [n]}$.
    To accommodate both settings simultaneously, let us define $j_0=1$ in robust regression, or $j_0=2$ in logistic regression, 
    so that 
    $\bP=\sum_{j=j_0}^p \be_j\be_j^T$ holds. Since $\sum_{i=1}^n\sum_{j=j_0}^p \| (\partial/\partial x_{ij}) \bpsi\|^2$ is upper bounded by $n C'$
    from \eqref{bound_frobenius_norm_derivatives}, the RHS of the inequality in \Cref{lm:chi_square}
    is $O(\sqrt{n})$. Therefore, \Cref{lm:chi_square} gives 
    $$
(p+1-j_0) \frac{\|\bpsi\|^2}{n}  - \frac{1}{n}\sum_{j=j_0}^p \Bigl(\bpsi^\top \bX \be_j - \sum_{i=1}^n \frac{\partial \be_i^\top \bpsi}{\partial x_{ij}}\Bigr)^2 = \Op(\sqrt{n}). 
    $$
    Here, $(p+1-j_0) \|\bpsi\|^2/n = p\|\bpsi\|^2/n+ \Op(1)$ by $\|\bpsi\|^2 = \Op(n)$, while 
    $\bpsi^\top \bX = \bm{0}_p^\top$ by the KTT condition. It remains to compute $\sum_{i=1}^n (\partial/\partial x_{ij})\be_i^\top \bpsi$. Using the derivative formula \eqref{eq:derivative_formula} and upper bounds \eqref{eq:as_bound}-\eqref{eq:moment_bound}, 
    \begin{align*}
        \sum_{j=j_0}^p \Bigl(\sum_{i=1}^n \frac{\partial \be_i^\top \bpsi}{\partial x_{ij}}\Bigr)^2 
        &= \|\bP \bA^\top \bX^\top \bD \bpsi + \tr[\bV]\bP \hat{\bbeta}\|^2 \\
        &= \tr[\bV]^2 \|\bP\hat\bbeta\|^2 + \|\bP\bA^\top\bX \bD \bpsi\|^2 + 2\tr[\bV] \hat{\bbeta}^\top \bP\bA^\top\bX^\top \bD \bpsi \\
        &= \tr[\bV]^2 \|\bP\hat{\bbeta}\|^2 + \Op(1) + \Op(n),
    \end{align*}
    which completes the proof. 
\end{proof}

\begin{lemma}\label{lm:relation_trV_trA_gamma}
    We have $ \tr[\bV] \tr[\bP\bA] = p + O(n^{1/2})$ and $ \tr[\bP\bA] \to^p \gamma$. 
\end{lemma}
\begin{proof}
    By the lemma above, we have
    $$
    \tr[\bV] \|\bP\hat{\bbeta}\|^2 - \tr[\bP\bA] \|\bpsi\|^2 = \Op(n^{1/2}), \qquad \|\bpsi\|^2 - p^{-1}\tr[\bV]^2\|\bP\hat{\bbeta}\|^2 = \Op(n^{1/2}). 
    $$
    Here, since $\|\bP\hat{\bbeta}\|^2 \to^p \sigma^2>0$ and $\|\bpsi\|^2/n\q \to^p \sigma^2/(\q\delta\gamma^2)$, the second display gives
    $
    \tr[\bV]/(\q n)\to^P 1/(\q\delta \gamma)
    $. On the other hand, substituting the second display to the first display, we are left with
    $$
    \tr[\bV]\|\bP\hat{\bbeta}\|^2 (1 - p^{-1}\tr[\bP\bA] \tr[\bV])  = \Op(n^{1/2}). 
    $$
    Since $\tr[\bV]\|\bP\hat{\bbeta}\|^2/n \to^P \sigma^2/(\delta\gamma^2) \cdot \sigma^2>0$, this gives $1 - p^{-1}\tr[\bP\bA] \tr[\bV]= \Op(n^{-1/2})$. Combined with $\tr[\bV]/(\q n)\to^p 1/(\q \delta \gamma)$, we have $\tr[\bP\bA] \to^p \gamma$. 
\end{proof}

\begin{lemma}\label{lm:sample_without_rep_psi}
    \assumLemma
    $$
    \frac{1}{|I\cap \tilde I|} \sum_{i\in I\cap \tilde{I}} \psi_i^2 = \frac{1}{|I|}\sum_{i\in I} \psi_i^2 + \Op(n^{-1/2}).  
    $$
\end{lemma}
\begin{proof}
    Let us use the following simple random sampling properties. 
\begin{lemma}[e.g., page 13 of \cite{chaudhuri2014modern}]\label{lm:sample_without_rep}
    Consider a deterministic array $(x_i)_{i=1}^M$ of length $M\ge 1$ and let $\mu$ be the mean $M^{-1}\sum_{i\in [M]} x_i$. 
    Suppose $J$ is uniformly distributed on  $\{J \subset [M]: |J|=m\}$ for a fixed integer $m\le M$. Then, the sample mean $\hat{\mu}_J =|J|^{-1}\sum_{i\in J} x_i$ is an unbiased estimate of the true mean $\mu$ and the variance is bounded as $\E[(\hat{\mu}_J-\mu)^2] \le  {\sum_{i\in M} x_i^2}/(m M)$. 
    \end{lemma}
    Recalling \Cref{rm:gepmetric_dist}, using \Cref{lm:sample_without_rep} with $m=|I\cap \tilde I|$ and $M=|I|$ conditionally on $(|\tilde I \cap I|, I, \bpsi)$ with $\sum_{i\in I} \psi_i^2 \le |I| C$ from \eqref{eq:as_bound} for a constant $C$, 
    $$
    \E\Bigl[\bigl| \frac{1}{|I\cap \tilde{I}|}\sum_{i\in I\cap \tilde{I}} \psi_i^2- \frac{1}{|I|} \sum_{i\in I}\psi_i^2 \bigr|^2\Bigr] \le \E\Bigl[\frac{\sum_{i\in I} \psi_{i}^2}{|I||I\cap \tilde I|}\Bigr] \le \frac{C}{|I\cap \tilde I|}. 
    $$
    Combined with the concentration $|I\cap \tilde{I}| = nq^2 + \op(n)$ (\Cref{rm:gepmetric_dist}), we complete the proof. 
\end{proof}

\begin{lemma}
    \label{eq:application_Gaussian_Poincare}
    Let $\bar{\E}[\cdot]=\E[\cdot|\bX\bbeta^*,\by]$ be the conditional expectation given $(\bX\bbeta^*, \by)$. \assumLemma 
    \begin{align*}
        \tilde{\bbeta}^\top\bP\hat{\bbeta}  = \bar{\E}\bigl[\tilde{\bbeta}^\top\bP\hat{\bbeta}\bigr] + \Op(n^{-1/2}), \quad 
        \frac{1}{|I\cap \tilde I|} \bpsi^\top \tilde{\bpsi}= \frac{1}{|I\cap \tilde I|} \bar{\E}\Bigl[
            \bpsi^\top \tilde{\bpsi}
        \Bigr] + \Op(n^{-1/2}).
    \end{align*}
\end{lemma}
\begin{proof}
    First we show the concentration of $|\tilde{\bbeta}\bP\hat{\bbeta}|$. 
    By the Gaussian Poincar\'e inequality with respect to $\bP \bX$, we have
    \begin{align*}
        \bar{\E}\bigl[
            \bigl( \tilde{\bbeta}^\top\bP\hat{\bbeta}  - \bar{\E}\bigl[\tilde{\bbeta}^\top\bP\hat{\bbeta}\bigr]\bigr)^2
        \bigr] 
        &\le             \sum_{j=1}^p \sum_{i=1}^n  \bar{\E} \bigl[
           (\frac{\partial \tilde{\bbeta}^\top \bP\hat{\bbeta}}{\partial x_{ij}})^2
        \bigr]\le 2             \sum_{j=1}^p \sum_{i=1}^n 
        \bar{\E}\Bigl[
(\tilde{\bbeta}^\top\bP \frac{\partial \hat\bbeta}{\partial x_{ij}})^2 + (\hat{\bbeta}^\top\bP \frac{\partial \tilde\bbeta}{\partial x_{ij}})^2
        \Bigr].
    \end{align*}
    By the symmetry of $\tilde{\bbeta},\hat{\bbeta}$, it suffices to bound $\sum_{j=1}^p \sum_{i=1}^n \bar{\E}\Bigl[(\tilde{\bbeta}^\top\bP \frac{\partial \hat\bbeta}{\partial x_{ij}})^2\Bigr]$. Using the derivative formula and the upper bounds in \Cref{lm:derivative_formula}, 
        \begin{align*}
        \sum_{j=1}^p \sum_{i=1}^n (\tilde{\bbeta}^\top\bP \frac{\partial \hat\bbeta}{\partial x_{ij}})^2
        &\le 2 \Bigl(\|\bA^\top\bP\tilde{\bbeta}\|^2 \|\bpsi\|^2 + \|\tilde{\bbeta}^\top\bP \bA \bX^\top \bD\|^2 \|\bbeta\|^2\Bigr),
        \end{align*}
        and the moment of the RHS is $O(n^{-1})$. This concludes the proof of concentration for $|\tilde{\bbeta}\bP\hat{\bbeta}|$. For $\tilde{\bpsi}^\top \tilde\bpsi$, the same argument using the Gaussian Poincar\'e inequality gives
        \begin{align*}
            \bar{\E}\Bigl[\Bigl(
              \bpsi^\top \tilde{\bpsi} - \bar{\E}\bigl[
                    \bpsi^\top \tilde{\bpsi}
                \bigr]
            \Bigr)^2\Bigr] 
            \le  2\sum_{j=1}^p \sum_{i=1}^n   \bar{\E} \Bigl[(\tilde{\bpsi}^\top \frac{\partial  \bpsi}{\partial x_{ij}})^2 + (\bpsi^\top \frac{\partial \tilde\bpsi}{\partial x_{ij}})^2 \Bigr].  
        \end{align*}
        Using the derivative formula and the upper bounds in \Cref{lm:derivative_formula} again, 
        \begin{align*}
           \sum_{j=1}^p \sum_{i=1}^n  (\tilde{\bpsi}^\top \frac{\partial \bpsi}{\partial x_{ij}})^2 
            &\le 2 \Bigl(\|\tilde\bpsi^\top \bD \bX\bA\|^2\|\bpsi\|^2 + \|\tilde{\bpsi}^\top \bV \|^2 \|\hat{\bbeta}\|^2\Bigr),
        \end{align*}
        and the moment of the RHS is $O(n)$. This gives 
        $
                     \bpsi^\top \tilde{\bpsi} - \bar{\E}\bigl[
                    \bpsi^\top \tilde{\bpsi}
                \bigr]= \Op(n^{1/2}).
        $
        Finally, dividing by $|I\cap \tilde I|=n\q^2 + \op(n)$ (see \Cref{rm:gepmetric_dist}), we obtain the concentration of $\bpsi^\top \tilde{\bpsi}$. 
\end{proof}

\section{Conclusion}\label{sec:conclusion}
This paper investigates the asymptotic behavior of bagging unregularized M-estimator for robust and logistic regression under the proportional high-dimensional regime. In particular, we have derived the new nonlinear system equation characterizing the limit of the risk of bagging estimators, revealing how the sub-sample size impacts the performance of the bagging estimator. Throughout the analysis, we assumed that the sub-samples are drawn without replacement. A natural direction for future work is to consider more general weighting schemes, as studied in 
\cite{peter1995learning, clarte2024analysis, karoui2018can}. Of particular interest is the analysis of risk for ensemble methods such as bagging (where we sample with replacement), or other random weighting
schemes where the data-fitting loss for the estimator $\hat{b}_m$ for each $m\in [M]$ is given by $\sum_{i=1}^n w_{m,i} \ell_{y_i}(\bx_i^\top \bb)$, where 
weights $(w_{m,i})_{m\in [M], i\in [n]}$ are sampled independently of the data $(X, y)$. Example includes the iid Poisson weights $w_{m,i} \sim  \text{Poisson}(1)$ (i.i.d.) for each $m\in[M]$ and $i\in [n]$, and independent multinomial weights $(w_{m,1}, \dots w_{m, n}) \sim \text{Multinomial}(n, n, n^{-1})$ for each $m\in [M]$.

\bibliographystyle{alpha}
\bibliography{reference.bib}

\appendix

\section{Additional numerical simulation for robust regression}

\subsection{Other noise distribution}\label{subsec:change_noise_dist}
We change the noise distribution to a $t$-distribution with $\text{df}=3$ and conducted the same experiment as in \Cref{fig:huber}. The additional simulation result is presented in \Cref{fig:huber_df3}, which suggests that the scale of the noise plays the same role in this setting as well. 
\begin{figure}[htbp]
    \centering
    \begin{subfigure}[b]{0.49\textwidth}
        \centering
        \includegraphics[width=\textwidth]{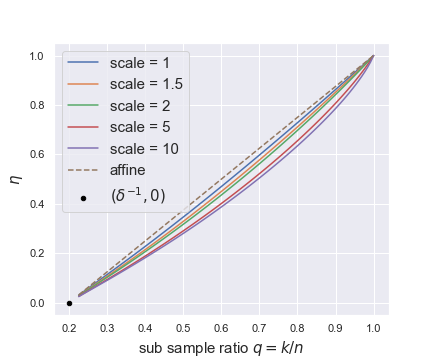}
    \end{subfigure}
    \begin{subfigure}[b]{0.49\textwidth}
        \centering
        \includegraphics[width=\textwidth]{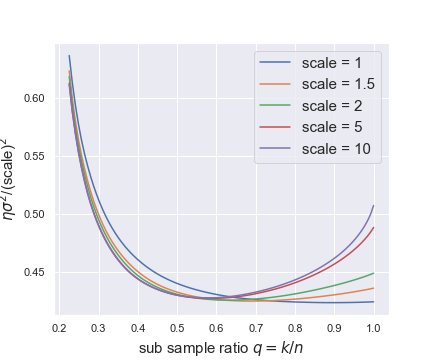}
    \end{subfigure}
    \begin{subfigure}[b]{0.49\textwidth}
        \centering
        \includegraphics[width=\textwidth]{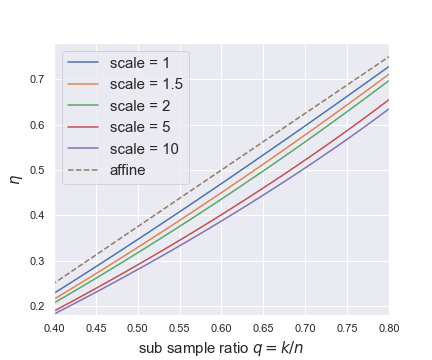}
    \end{subfigure}
    \begin{subfigure}[b]{0.49\textwidth}
        \centering
        \includegraphics[width=\textwidth]{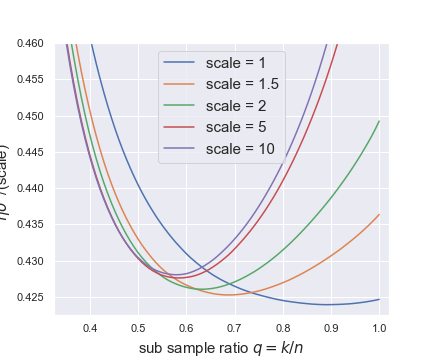}
    \end{subfigure}
    \caption{Plot of ${\q} \mapsto \eta$ and ${\q}\mapsto \sigma^2\eta$
        obtained by solving \eqref{eta_equation_robust} numerically.
        {Different noise distributions are given by} $(\text{scale})\times \text{t-dist (df=3)}$, for scale$\in\{1,{ 1.5, 2, 5,} 10\}$. The dashed line is the affine line $\q\mapsto (q-\delta^{-1})/(1-\delta^{-1})$. 
    {The bottom plots zoom in on a specific region of the top plots.}
    }
    \label{fig:huber_df3}
\end{figure}

\subsection{Pseudo Huber loss}\label{subsec:pseudo_huber}
We adopt the pseudo-Huber loss $\sqrt{1 + x^2}$, which satisfies \Cref{assum:robust}, and replicate the experiment shown in \Cref{fig:huber_compare}. The results, presented in \Cref{fig:pseudo_huber}, further support the validity of \Cref{thm:robust}.

\begin{figure}[htbp]
    \centering
    \begin{subfigure}[b]{0.49\textwidth}
        \centering
        \includegraphics[width=\textwidth]{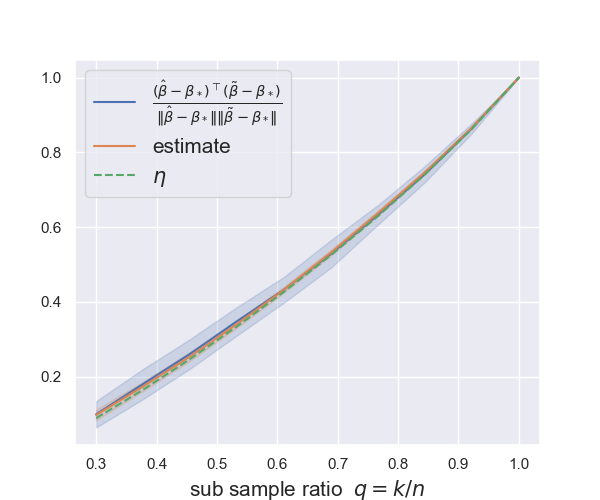}
    \end{subfigure}
    \begin{subfigure}[b]{0.49\textwidth}
        \centering
        \includegraphics[width=\textwidth]{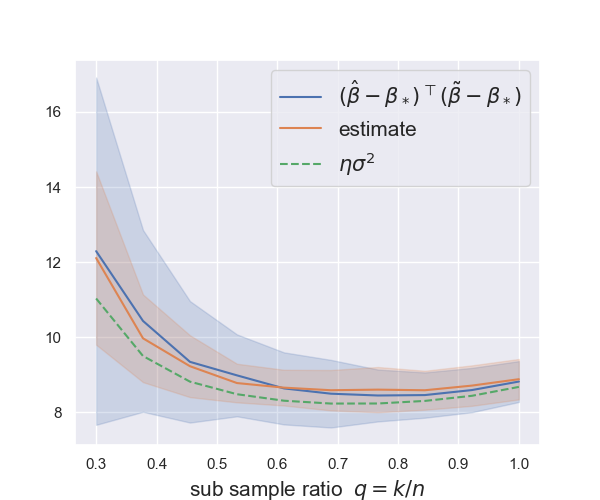}
    \end{subfigure}
    \caption{Comparison of simulation results, theoretical curves
            obtained by solving \eqref{eta_equation_robust} numerically,
    and estimate constructed by \eqref{eq:thm_estimation}, for the pseudo Huber loss $\rho(x)=\sqrt{1+x^2}$. 
     Here, the noise distribution is fixed to $4\times \text{t-dist(df=2)}$ and $(n, p)=(5000, 1000)$. The error bar is standard deviation with $10$ Monte Carlo simulations. 
     }
    \label{fig:pseudo_huber}
\end{figure}

\subsection{Small sample size experiments}
We conducted additional simulation about the robust regression for $n=500, 1000$. \Cref{fig:huber_compare_n_change} suggests that the correlation $(\hat{\bb}-\bbeta_*)^\top (\tilde{\bb}-\bbeta_*)/\|\hat{\bb}-\bbeta_*\|_2 \|\tilde{\bb}-\bbeta_*\|$ is still approximated well by the deterministic solution $\eta$ to the nonlinear system and the estimator \eqref{eq:thm_estimation}. 

\begin{figure}[htbp]
    \centering
    \begin{subfigure}[b]{0.49\textwidth}
        \centering
        \includegraphics[width=\textwidth]{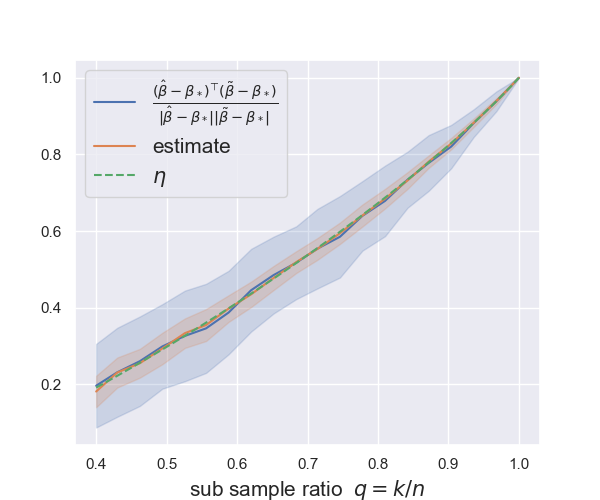}
    \end{subfigure}
    \begin{subfigure}[b]{0.49\textwidth}
        \centering
        \includegraphics[width=\textwidth]{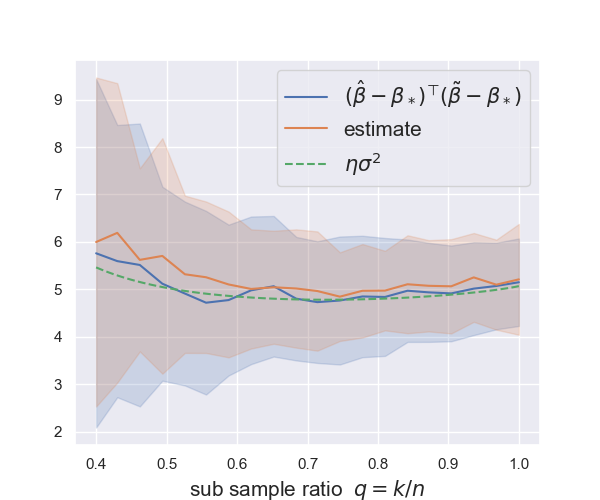}
    \end{subfigure}

    \begin{subfigure}[b]{0.49\textwidth}
        \centering
        \includegraphics[width=\textwidth]{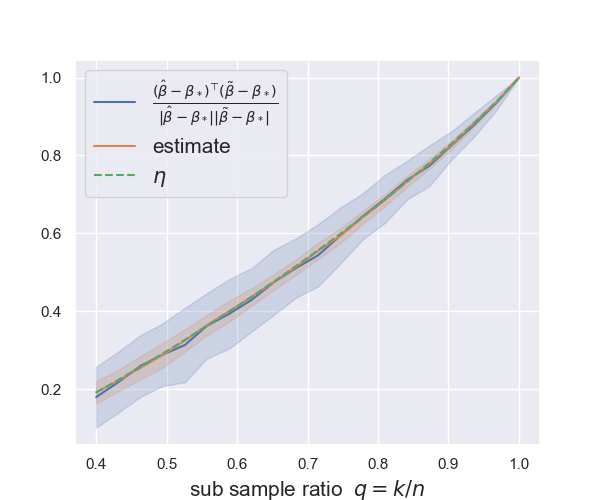}
    \end{subfigure}
    \begin{subfigure}[b]{0.49\textwidth}
        \centering
        \includegraphics[width=\textwidth]{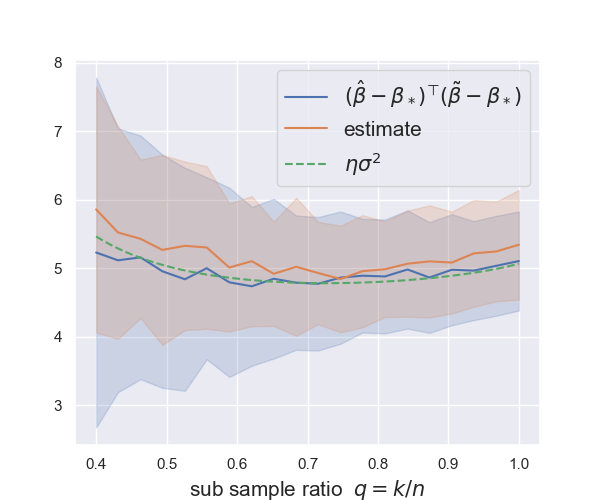}
    \end{subfigure}
    \caption{Comparison of simulation results, theoretical curves
            obtained by solving \eqref{eta_equation_robust} numerically,
    and estimate constructed by \eqref{eq:thm_estimation}. Here, the noise distribution is fixed to $3\times \text{t-dist(df=2)}$. 
     $(n, p)=(500, 100)$ in the top row and $(n, p)=(1000, 200)$ in the bottom row. The error bar is standard deviation with $100$ Monte Carlo simulation. 
     }
    \label{fig:huber_compare_n_change}
\end{figure}

\subsection{Universality}
We have added additional simulations in \Cref{fig:huber_compare_universality} to further examine the universality phenomenon, suggesting that \Cref{thm:robust} continues to hold across various non-Gaussian covariate distributions.

\begin{figure}[htbp]
    \centering
    \begin{subfigure}{0.3\textwidth}
        \centering
        \includegraphics[width=\textwidth]{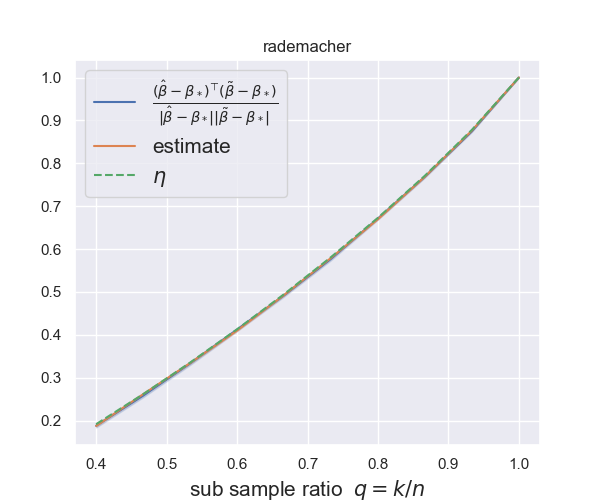}
    \end{subfigure}
    \hfill
    \begin{subfigure}{0.3\textwidth}
        \centering
        \includegraphics[width=\textwidth]{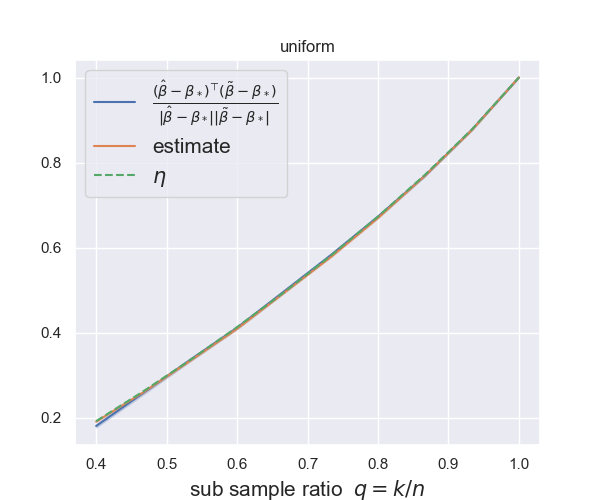}
    \end{subfigure}
    \hfill
    \begin{subfigure}{0.3\textwidth}
        \centering
        \includegraphics[width=\textwidth]{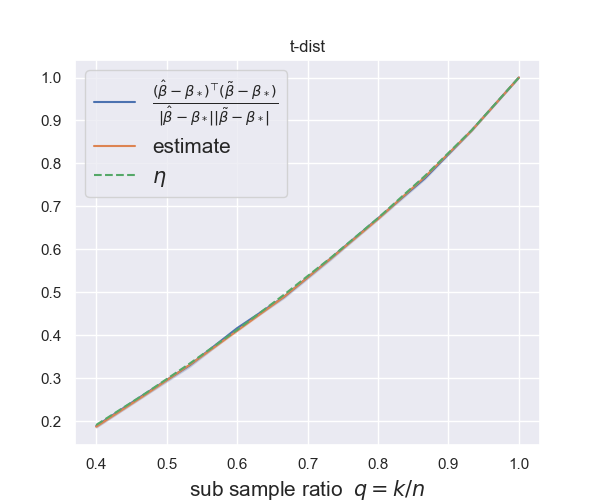}
    \end{subfigure}

    \begin{subfigure}{0.3\textwidth}
        \centering
        \includegraphics[width=\textwidth]{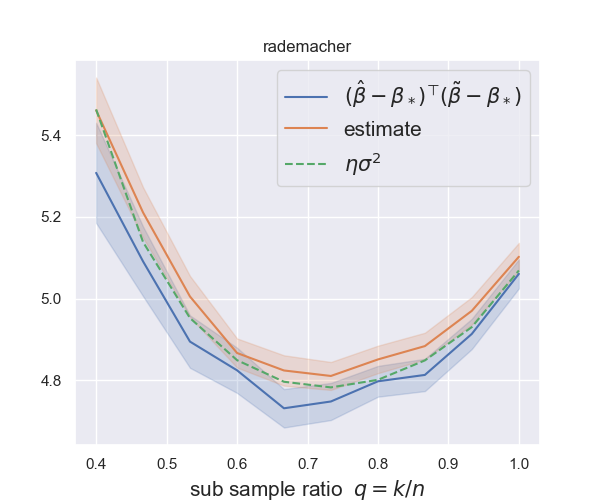}
    \end{subfigure}
    \hfill 
    \begin{subfigure}{0.3\textwidth}
        \centering
        \includegraphics[width=\textwidth]{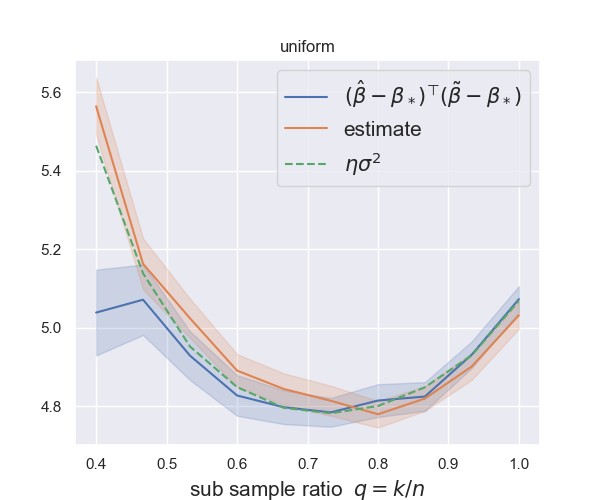}
    \end{subfigure}
    \hfill
    \begin{subfigure}{0.3\textwidth}
        \centering
        \includegraphics[width=\textwidth]{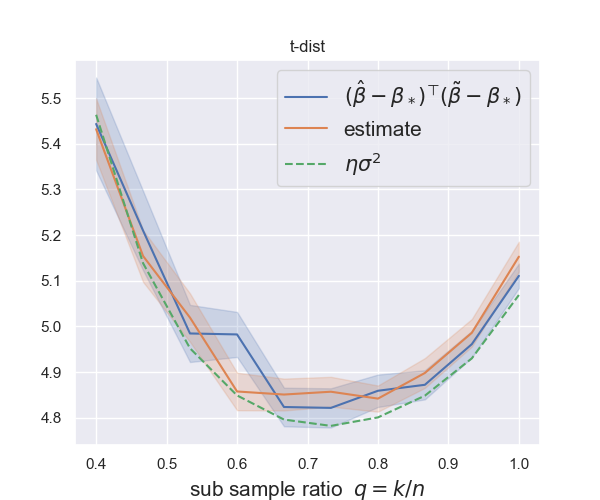}
    \end{subfigure}

    \caption{Comparison of simulation results, theoretical curves
            obtained by solving \eqref{eta_equation_robust} numerically,
    and estimate constructed by \eqref{eq:thm_estimation}. 
    The distribution of the covariate $X$ is set to Rademacher, Uniform, and t-distribution with $\text{df}=4$ (from left to right), normalized to match the first and second moments of $N(0,1)$. The sample size and feature dimension are fixed at 
    $(n, p)$ is fixed to $(5000, 1000)$, and the noise distribution follows t-distribution with $\text{df}=2$. 
    }
    \label{fig:huber_compare_universality}
\end{figure}

\section{Additional numerical simulation for logistic regression}\label{sec:additional_simulation_logit}
We examine the theoretical risk limit $\sigma^2 \eta$ obtained by \eqref{eta_equation_logi} for 
large aspect ratios $\delta = n/p \in \{15, 20, 25, 30\}$ across various signal strengths $|\bbeta_*| \in \{0, 0.1, 0.2, 0.3, 0.4\}$. As shown in \Cref{fig:logit_ushape_fixedpoint}, for $\delta > 20$, the risk curve in $q = k/n$ exhibits a U-shape, highlighting the benefit of subsampling for risk reduction.

\begin{figure}[htpb]
    \centering

    \includegraphics[width=\textwidth]{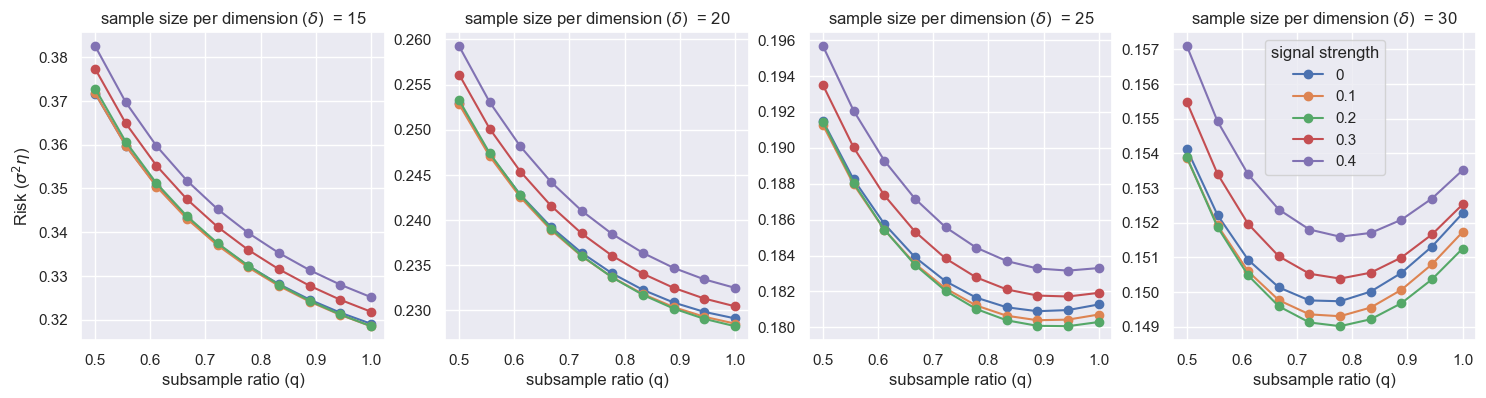}

    \caption{
        The theoretical curves of $\q\mapsto \sigma^2 \eta$
        obtained by solving \eqref{eta_equation_logi} numerically for varying values of the aspect ratio $\delta (=\lim n/p)$ and signal strength $\|\bbeta_*\|$.
        }
    \label{fig:logit_ushape_fixedpoint}
\end{figure}

\end{document}